\documentclass[a4paper,12pt, british, reqno]{amsart}

\usepackage{babel}
\usepackage{mathabx}
\usepackage{amssymb}
\usepackage{amsfonts}
\usepackage{enumitem}
\usepackage{hyperref}
\usepackage[utf8]{inputenc}
\usepackage{newunicodechar}
\usepackage{mathtools}
\usepackage{varioref}
\usepackage[centering]{geometry}
\usepackage[arrow,curve,matrix]{xy}
\usepackage{csquotes}
\usepackage{dynkin-diagrams}
\usepackage{extarrows}

\usepackage{colortbl}
\usepackage{graphicx}
\usepackage{tikz}
\usepackage{tikz-cd}

\usepackage{mathrsfs}

\usepackage{multirow}
\usepackage{multicol}
\usepackage{longtable} 
\usepackage{caption}
\usepackage{hhline}
\newcolumntype{M}[1]{>{\centering\arraybackslash}m{#1}} 

\definecolor{linkred}{rgb}{0.7,0.2,0.2}
\definecolor{linkblue}{rgb}{0,0.2,0.6}

\numberwithin{figure}{section}

\setdescription{labelindent=\parindent, leftmargin=2\parindent}
\setitemize[1]{labelindent=\parindent, leftmargin=2\parindent}
\setenumerate[1]{labelindent=0cm, leftmargin=*, widest=iiii}

\DeclareFontFamily{OMS}{rsfs}{\skewchar\font'60}
\DeclareFontShape{OMS}{rsfs}{m}{n}{<-5>rsfs5 <5-7>rsfs7 <7->rsfs10 }{}
\DeclareSymbolFont{rsfs}{OMS}{rsfs}{m}{n}
\DeclareSymbolFontAlphabet{\scr}{rsfs}
\DeclareSymbolFontAlphabet{\scr}{rsfs}

\DeclareMathOperator{\Ad}{Ad}

\DeclareMathOperator{\codim}{codim}

\DeclareMathOperator{\Pic}{Pic}
\DeclareMathOperator{\rank}{rank}

\DeclareMathOperator{\reg}{reg}

\DeclareMathOperator{\sing}{sing}
\DeclareMathOperator{\Spec}{Spec}

\DeclareMathOperator{\gr}{gr}

\newcommand{\sO}{\scr{O}}

\newcommand{\cD}{\mathcal D}

\newcommand{\cO}{\mathcal O}

\newcommand{\cY}{\mathcal Y}

\newcommand{\bC}{\mathbb{C}}

\newcommand{\bO}{\mathbb{O}}
\newcommand{\bP}{\mathbb{P}}
\newcommand{\bQ}{\mathbb{Q}}

\newcommand{\bZ}{\mathbb{Z}}

\newcommand{\bfa}{\mathbf{a}}

\newcommand{\bfl}{\mathbf{l}}

\newcommand{\bfm}{\mathbf{m}}

\newcommand{\bfn}{\mathbf{n}}

\newcommand{\bfO}{\mathbf{O}}

\newcommand{\bfr}{\mathbf{r}}

\newcommand{\bfu}{\mathbf{u}}
\newcommand{\bfv}{\mathbf{v}}
\newcommand{\bfx}{\mathbf{x}}
\newcommand{\bfy}{\mathbf{y}}
\newcommand{\bfZ}{\mathbf{Z}}

\newcommand{\fc}{\mathfrak{c}}

\newcommand{\fg}{\mathfrak{g}}

\newcommand{\fl}{\mathfrak{l}}

\newcommand{\fp}{\mathfrak{p}}

\newcommand{\ft}{\mathfrak{t}}
\newcommand{\fu}{\mathfrak{u}}

\newcommand{\fz}{\mathfrak{z}}

\theoremstyle{plain}
\newtheorem{thm}{Theorem}[section]

\newtheorem{cor}[thm]{Corollary}

\newtheorem{lem}[thm]{Lemma}

\newtheorem{prop}[thm]{Proposition}

\theoremstyle{remark}

\newtheorem{c-n-d}[thm]{Claim and Definition}

\newtheorem{rem}[thm]{Remark}

\newtheorem*{rem-nonumber}{Remark}

\numberwithin{equation}{thm}

\setlist[enumerate]{label=(\thethm.\arabic*),
	before={\setcounter{enumi}{\value{equation}}},
	after={\setcounter{equation}{\value{enumi}}}}

\newcommand{\factor}[2]{\left. \raise 2pt\hbox{$#1$} \right/\hskip -2pt\raise
	-2pt\hbox{$#2$}}

\author{Baohua Fu}
\address{Baohua Fu, State Key Laboratory of Mathematical Sciences, Morningside
	Center of Mathematics, Academy of Mathematics and Systems Science, Chinese
	Academy of Sciences, Beijing 100190, China;   and School of Mathematical
	Sciences, University of Chinese Academy of Sciences, Beijing, China}
\email{\href{bhfu@math.ac.cn}{bhfu@math.ac.cn}}
\urladdr{\href{http://www.math.ac.cn/people/fbh/}{http://www.math.ac.cn/people/fbh/}}

\author{Jie Liu} %
\address{Jie Liu, Institute of Mathematics, Academy of Mathematics and Systems
	Science, Chinese Academy of Sciences, Beijing, 100190, China}
\email{\href{jliu@amss.ac.cn}{jliu@amss.ac.cn}}
\urladdr{\href{http://www.jliumath.com}{http://www.jliumath.com}}


\makeatletter
\@namedef{subjclassname@2020}{2020 Mathematics Subject Classification}
\makeatother

\title[]{A connection between minimal nilpotent orbits of types A and D via Hamiltonian reduction}
\date{\today}

\makeatletter

\hypersetup{
	pdfauthor={\authors},
	pdftitle={\@title},
	pdfsubject={\@subjclass},
	pdfkeywords={\@keywords},
	pdfstartview={Fit},
	pdfpagemode={UseOutlines},
	pdfpagelayout={OneColumn},
	colorlinks,
	linkcolor=linkblue,
	citecolor=linkred,
	urlcolor=linkred
}
\makeatother

\DeclareMathOperator{\aff}{aff}

\DeclareMathOperator{\Proj}{Proj}

\begin{document}
	
	\begin{abstract}
    We establish a novel connection between the minimal nilpotent orbit  \(\mathbb{O}_n\) in \(\mathfrak{sl}_n\) and the minimal nilpotent orbit closure \(\overline{\mathbf{O}}_n\) in \(\mathfrak{so}_{2n+2}\), which differs from the shared‑orbit paradigm of Brylinski and Kostant, where no direct type‑A--type‑D relation appears.  More precisely, we show that the affine closure of the cotangent bundle \(\overline{T^*\mathbb{O}_n}^{\mathrm{aff}}\) is isomorphic to a \(\mathbb{C}^*\)-Hamiltonian reduction of  \(\overline{\mathbf{O}}_n\). This provides a quasi-classical analogue of a quantum result of Levasseur and Stafford. A detailed study of the geometry of this Hamiltonian reduction reveals that \(\overline{T^*\mathbb{O}_n}^{\mathrm{aff}}\) has no symplectic resolution.
	\end{abstract}

	\maketitle
\tableofcontents

\section{Introduction}

Nilpotent orbits constitute a central class of objects in Lie theory. Every nilpotent orbit in a semisimple complex Lie algebra carries the classical Kostant--Kirillov--Souriau symplectic form, which naturally endows it with the structure of a homogeneous symplectic manifold. This symplectic structure not only closely links nilpotent orbits to geometric quantization and the orbit method in representation theory, but also provides a geometric framework for understanding the coadjoint representation of Lie algebras. Moreover, the closures of nilpotent orbits often exhibit rich and interesting singularities.
    
Closely related to the singularities of nilpotent orbit closures is the theory of symplectic singularities, introduced by Beauville \cite{Beauville2000a}. A normal irreducible variety $Y$ is called \emph{symplectic} if there exists a symplectic form $\omega$ on $Y_{\reg}$  such that for any resolution $f\colon W\rightarrow Y$, the pull-back $f^*\omega$  extends to a holomorphic $2$-form on $W$. Furthermore, a resolution $f$ is called \emph{symplectic} if $f^*\omega$ extends to a symplectic form on $W$. Symplectic singularities appear naturally in moduli spaces, representation theory, and mirror symmetry. A fundamental problem in this field is to construct and classify symplectic singularities and to study the existence of symplectic resolutions.

Traditionally, research has focused primarily on nilpotent orbits inside a single fixed Lie algebra. In \cite{Levasseur-Smith}, a surprising connection between nilpotent orbits in different Lie algebras was discovered: the normalization of the $8$-dimensional nilpotent orbit closure in $\fg_2$ is in fact isomorphic to the minimal nilpotent orbit closure in $\mathfrak{so}_7$. In a subsequent celebrated work of Brylinski and Kostant \cite{Brylinski-Kostant}, this connection was investigated in depth, and all pairs of Lie algebras $\fg_0\subset \fg$  with $\fg$ simple for which there exist nilpotent orbits $\cO_0\subset \fg_0$  and $\cO\subset \fg$ admitting a finite $G_0$ -equivariant morphism from $\cO_0$ to $\cO$ were classified. The complete classification of all such {\rm shared nilpotent orbits} $(\cO_0,\cO)$ has recently been obtained in \cite[Proposition~2.12]{FJLS}.

Against this background, one motivation of the present paper is to establish a novel connection between the minimal nilpotent orbits of type A and type D, which is of a different nature from the shared-orbit paradigm of Brylinski and Kostant. Notably, even within the theory of shared nilpotent orbits, no such connection between type A and type D exists.

Another motivation of this paper is to investigate the existence of symplectic resolutions for the affine closure of the cotangent bundle of the minimal nilpotent orbit $\bO_n$ in $\mathfrak{sl}_n$. Recall that the \emph{affine closure} $\overline{Y}^{\aff}$ of an irreducible variety $Y$ is defined as the affine scheme $\Spec \sO(Y)$, where $\sO(Y)$ is the ring of regular functions on $Y$. It is proven in \cite[Theorem 1.2]{FuLiu2025} that the affine closure $\overline{T^*\bO_n}^{\aff}$ is a symplectic variety; it is therefore natural to ask whether $\overline{T^*\bO_n}^{\aff}$ admits a symplectic resolution. One aim of this paper is to answer this question in the negative:

\begin{thm}
\label{t.TOn-symp-res}
    Assume that $n\geq 2$. The affine closure  $\overline{T^*\bO_n}^{\aff}$ is a symplectic variety with terminal singularities, and has no symplectic resolutions.
\end{thm}

More precisely, we explicitly describe the analytic germ of $\overline{T^*\bO_n}^{\aff}$ at every point except one. Theorem~\ref{t.TOn-symp-res} then follows directly from this local description.  Furthermore, we prove that $\overline{T^*\bO_n}^{\aff}$ is not $\bQ$-factorial if and only if $n\geq 3$ (see Corollary~\ref{cor.Q-factoriality}), and that it admits exactly two  $\bQ$-factorial terminalizations (see Corollary \ref{cor.Q-fact-ter}), which are described in detail in \S\,\ref{ss.Q-fact-term}.

In general, it is a notoriously difficult problem to explicitly describe the geometry of the affine closure of a quasi-affine variety along its boundary. To overcome this difficulty in Theorem~\ref{t.TOn-symp-res}, we realize $\overline{T^*\bO_n}^{\aff}$ as a Hamiltonian reduction of the minimal nilpotent orbit closure $\overline{\bfO}_n$ in $\mathfrak{so}_{2n+2}$, thereby establishing a connection between the minimal nilpotent orbits of type~A and type~D. More precisely, consider the $\bC^*$-subgroup of $\textup{SO}_{2n+2}$ (written in the matrix form of \eqref{eq.Form-xi}) given by
\[
{\rm diag}(1, \lambda, \cdots, \lambda, 1, \lambda^{-1}, \cdots, \lambda^{-1}).
\]
which acts on $\mathfrak{so}_{2n+2}$ via the adjoint action. This action is Hamiltonian with respect to the Kostant–Kirillov–Souriau form on $\bfO_n$, and therefore one can define the Hamiltonian reduction
\[
\overline{\bfO}_n/\!\!/\!\!/\bC^* \coloneqq \mu^{-1}(0)/\!/\bC^*,
\]
where $\mu\colon \overline{\bfO}_n \to \fc^* \cong \bC$ is the associated moment map (see \cite{FuLiu2026}). We prove:

\begin{thm}
\label{thm.main-thm}
    Assume that $n\geq 3$. Then $\overline{T^*\bO_n}^{\aff}$ is isomorphic to the Hamiltonian reduction $\overline{\bfO}_n/\!\!/\!\!/\bC^*$.
\end{thm}

This result can be regarded as the quasi-classical analogue of \cite[Theorem 0.1]{LevasseurStafford1999}. We analyze in detail the local geometry of the quotient $\overline{\bfO}_n/\!\!/\!\!/\bC^*$ away from its vertex (with respect to the descent of the natural scaling on $\overline{\bfO}_n$), from which we derive Theorem \ref{t.TOn-symp-res}. Moreover, this result allows us to use the variation of GIT to construct the $\bQ$-factorial terminalizations of $\overline{T^*\bO_n}^{\aff}$. We refer the reader to \S\,\ref{s.HamRed} and Theorem \ref{thm.M-Geometry} for more details. During the preparation of this paper, B.~Jia informed us that he has independently obtained Theorem \ref{thm.main-thm} (via a different construction).

\begin{rem}
Let $\fg$ be a simple Lie algebra and let $\cO\subset \fg$ be the minimal nilpotent orbit. By \cite[Theorem~1.2]{FuLiu2025}, the affine closure $\overline{T^*\cO}^{\aff}$ is a symplectic variety. If $\fg$ is of type C, i.e., $\fg=\mathfrak{sp}_{2n}$, then $\overline{\cO}^{\aff}$ is isomorphic to $\bC^{2n}/\langle\pm 1\rangle$, so the symplectic variety $\overline{T^*\cO}^{\aff}$  is in fact isomorphic to $\bC^{4n}/\langle \pm 1\rangle$ (see \cite[Example~2.8]{FuLiu2025}). It remains a very interesting problem to describe the singularities and to study the existence of symplectic resolutions of $\overline{T^*\cO}^{\aff}$ in the remaining types. We expect that for all other types, the affine closure $\overline{T^*\cO}^{\aff}$ is terminal and $\bQ$-factorial, and admits no symplectic resolutions.
\end{rem}

 \subsection*{Acknowledgments}
 We would like to thank Gwyn Bellamy for his helpful remarks on the variation of GIT. Both authors are supported by the National Key Research and Development Program of China (No. 2025YFA1017302),  the CAS Project for Young Scientists in Basic Research (No. YSBR-033) and the NSFC grant (No. 12288201). This work was supported by the Strategic Priority Research Program of Chinese Academy of Sciences under Grant XDA0480503. 

\section{Symplectic structure on some nilpotent orbits}

In this section, we recall a result from \cite[\S\,4]{FuLiu2025}, where it was shown that some nilpotent orbit closures can be realized as the affine closure of the cotangent bundle of certain smooth quasi-affine varieties. The aim of this section is to sharpen this result by showing that these isomorphisms are symplectic with respect to the natural symplectic forms on both sides.

\subsection{Nilpotent orbits as affine closures}

Let $G$ be a connected simple group, and let $P \subset G$ be a maximal parabolic subgroup containing a Borel subgroup $B$. Let $\mathfrak{h} \subset \mathfrak{b}$ be the Cartan subalgebra. Write $\mathfrak{p} = \mathfrak{l} \oplus \mathfrak{u}^+$, where $\mathfrak{u}^+$ is the nilpotent radical of $\mathfrak{p}$ and $\mathfrak{l}$ is the Levi part. Thus, $\mathfrak{g} = \mathfrak{u}^+ \oplus \mathfrak{l} \oplus \mathfrak{u}^-$. Write $g=g^+ + g^l + g^-$ the corresponding decomposition of $g\in \fg$. We will consider certain  specified $\fp$ with abelian $\mathfrak{u}^+$, as summarized in the following table (see also \cite[\S\,4]{FuLiu2025}). 
		\begin{longtable}{llp{3.1cm}l}
			\caption{}\label{table.Abelian}\\
			\hline
			\text{Type of $\mathfrak{g}$} 
			& \text{Dynkin diagram} 
			& parabolics $\mathfrak{p}$ with abelian radical
			& nilpotent orbit $\mathcal{O}$ \\
			\hline
			\endfirsthead
			\multicolumn{4}{l}{{ {\bf \tablename\ \thetable{}} \textrm{-- continued from previous page}}}
			\\
			\hline 
			\text{Type of $\mathfrak{g}$} 
			& \text{Dynkin diagram} 
			& parabolics $\mathfrak{p}$ with abelian radical
			& nilpotent orbit $\mathcal{O}$
			\endhead
			\hline
			\hline \multicolumn{4}{r}{{\textrm{Continued on next page}}} \\ \hline
			\endfoot
			
			\hline \hline
			\endlastfoot
			$A_n$ $(n \geq 1)$ 
			& $\dynkin[labels={1,2, ,n}] A{*2.*2}$
			& $\mathfrak{p}_i$ $\left(2 \leq i \leq \frac{n+1}{2}\right)$ 
			& $\mathcal{O}_r$ $(1 \leq r \leq i-1)$ \\
			$B_n$ $(n \geq 3)$ 
			& $\dynkin[labels={1,2,,n}] B{*2.*2}$
			& $\mathfrak{p}_1$ 
			& $\mathcal{O}_{\min}$ \\
			$C_n$ $(n \geq 2)$
			& $\dynkin[labels={1,2,,n}] C{*2.*2}$
			& $\mathfrak{p}_n$
			& $\mathcal{O}_r$ $(1 \leq r \leq n - 1)$ \\
			$D_n$ $(n \geq 4)$
			& $\dynkin[labels={1,2,,n-1,n}] D{*2.*3}$
			& $\mathfrak{p}_1$, $\mathfrak{p}_{n-1}$, $\mathfrak{p}_n$
			& $\mathcal{O}_{\min}$ \\
			$D_n$ $(n \geq 4)$
			& $\dynkin[labels={1,2,,n-1,n}] D{*2.*3}$
			& $\mathfrak{p}_n$
			& $\mathcal{O}_{r}$ $\left(2 \leq r \leq \frac{n - 2}{2}\right)$ \\
			
			$E_6$ 
			& $\dynkin[labels={1,...,6}] E6$
			& $\mathfrak{p}_1$, $\mathfrak{p}_6$
			& $\mathcal{O}_{\min}$   \\
			$E_7$
			& $\dynkin[labels={1,...,7}] E7$
			& $\mathfrak{p}_7$
			& $\mathcal{O}_{\min}$ \\
			\hline
		\end{longtable}
The extra notation in Table \ref{table.Abelian} is as follows. In each case, $\mathfrak{p}_i$ is the maximal parabolic subalgebra of $\mathfrak{g}$ associated to the simple root $\alpha_i$ (in the order of Bourbaki), and $\mathcal{O}_{\min}$ denotes the minimal nilpotent orbit in $\mathfrak{g}$. In Cases $A_n$, $C_n$, and $D_n$, the variety $\mathcal{O}_r$ is a nilpotent orbit in $\mathfrak{g}$ defined as follows (see \cite[II, \S\,5 and \S\,6]{LevasseurStafford1989}):
	\[
	\mathcal{O}_r \coloneqq \{g \in \mathfrak{g} \mid \rank(g) = \widetilde{r} \text{ and } g^2 = 0\},
	\]
	where $\widetilde{r} = r$ in Cases $A_n$ and $C_n$, but $\widetilde{r} = 2r$ in Case $D_n$. Once $\fp$ and $\cO$ are given, we define $X^+$ (resp. $X^-$) to be $\cO\cap \fu^+$ (resp. $\cO\cap \fu^-$). Let $\kappa$ be the Killing form on $\fg$, then we have
$\kappa(\fu^+, \fu^+ \oplus \fl)=0$ and $\kappa(\fu^-, \fu^- \oplus \fl)=0$ in this case, which  in particular implies that
$\kappa(\fl, \fu^+ \oplus \fu^-)=0$ and $\kappa$ induces an isomorphism $\fu^+\cong \fu^-$ sending $X^+$ to $X^-$ isomorphically. Moreover, we have 
\[
\overline{\cO}_r^{\aff}\cong \overline{\cO}_r,\quad \overline{X^+}^{\aff}\cong \overline{X^+}\quad \text{and}\quad \overline{X^-}^{\aff}\cong \overline{X^-}
\]
where the right hand side in the isomorphisms above means taking closure in $\fg$. With these notations, we recall:
       
    \begin{thm}[\protect{\cite[Theorem 4.2]{FuLiu2025}}]
		\label{thm.Quadric}
        There exist two fibrations on $\overline{\cO}_r$
        \begin{equation}
            \label{eq.Fib-O}
            \begin{tikzcd}[row sep=large, column sep=large]
            \overline{\cO}_r \arrow[r,"{\pi_+}"] \arrow[d,"{\pi_-}" left]
                &   \overline{X^+} \subset \fu^+ \\
            \overline{X^-} \subset\fu^-
        \end{tikzcd}
        \end{equation}
       such that
        \begin{enumerate}
            \item $\pi_+^{-1}(X^+)$ (resp. $\pi_-^{-1}(X^-)$) is isomorphic to $T^*X^+$ (resp. $T^*X^-$), and 

            \item the affine closures $\overline{T^*X^+}^{\aff}$ and $\overline{T^*X^-}^{\aff}$ are isomorphic to $\overline{\cO}_r$,
        \end{enumerate}
        where $\pi_+$ and $\pi_-$ are induced by the decomposition $\fg=\fu^+\oplus \fl\oplus \fu^-$.
	\end{thm}
In the sequel of this section, we will denote $\cO_r$ by $\cO$ for simplicity. Recall that $\cO$ carries a natural (algebraic) symplectic form $\omega^{KKS}$, called the \emph{Kostant--Kirilov--Souriau form}. More precisely, for any $o\in \cO$, the tangent space $T_o \cO$ is isomorphic to $[\fg,o]$ so that $\omega^{KKS}$ is defined as follows:
\[
\omega_o^{KKS}([g,o],[g',o])\coloneqq \kappa(o,[g,g'])\quad g,\;g'\in \fg.
\]
On the other hand, for any smooth variety $Z$, its cotangent bundle $T^*Z$ also carries a natural symplectic form $d\eta$, where $\eta$ is the Liouville $1$-form defined as follows:
\[
\eta_{(z,w)} (\zeta)\coloneqq \langle w,d\pi(\zeta)\rangle\quad z\in Z,\; w\in T^*_z Z,\; \zeta\in T_{(z,w)}(T^*Z),
\]
where $\pi\colon T^*Z\rightarrow Z$ is the natural projection. From the proof of \cite[Theorem 4.2]{FuLiu2025}, it is known that under the isomorphisms in Theorem \ref{thm.Quadric}, the symplectic form on $\cO$ coincides with that on the cotangent bundle along its zero section, so it is natural to compare these two symplectic structures globally. We prove:

\begin{prop} \label{p.symiso}
Let $X$  be either $X^+$  or $X^-$ , and accordingly let $U$  be $\pi_+^{-1}(X^+)$ or $\pi_-^{-1}(X^-)$. Then the isomorphism $U\cong T^*X$  in Theorem~\ref{thm.Quadric} can be chosen to be symplectic.
\end{prop}

By the symmetry between $\fu^+$ and $\fu^-$, we may assume that $X=X^+$. Firstly we describe in details the isomorphism $T^*X\cong U$. Let $E$ be the vector subbundle of $X\times \fl$ defined as follows:
\[
E\coloneqq \left\{(x,[\fu^-,x])\mid x\in X\right\} \subset X\times \fl.
\]
Consider the following surjective map
\[
X\times \fu^- \longrightarrow U\subset \cO\subset  \fg,\quad (x,v)\longmapsto o,
\]
where
\[
o \coloneqq  \Ad(\exp(v))x = x + [v,x] + \tfrac12[v,[v,x]].
\]
Here we use the fact that $\fu^-$ is abelian so that $\textup{ad}_v^3=0$. Then the natural map 
\[
X\times \fu^-\longrightarrow E,\quad (x,v)\longmapsto (x,[v,x])
\]
induces an isomorphism $U\cong E$ (over $X$). As $x\in \fu^+$, we have $[\fu^+,x]=0$, which implies
\begin{equation}
\label{eq.TxO}
    T_x \cO \cong [\fg,x] = [\fl,x]\oplus [\fu^-,x].
\end{equation}
As $[\fl, x] \subset \fu^+$, $[\fu^-, x] \subset \fl$ and $T_x X \subset \fu^+$, we get $[\fl, x] = T_xX$ so that $[\fu^-,x]$ can be identified with $T_x^* X$ as follows: For any $v\in \fu^-$, we define a linear map
\begin{equation}
\label{eq.E-to-Cotang}
    [v,x]\colon [\fl,x]\cong T_x X\longrightarrow \bC,\quad [l,x]\longmapsto \omega^{KKS}_x([l,x],[v,x])=\kappa(x,[l,v]).
\end{equation}
This defines an isomorphism $E\cong T^*X$ and, hence, an isomorphism $U\cong T^*X$. 

\subsection{Proof of Proposition \ref{p.symiso}}

In the following, we prove Proposition \ref{p.symiso} by showing that the isomorphism $U\cong T^*X$ introduced in the last section is  actually symplectic. To this end, we consider the associated $1$-forms of the two symplectic forms and prove that their difference is an exact $1$-form.

Let $\Ad(\exp(v))x=o\in U$. Consider a curve $o(t)=\Ad(\exp(v(t)))x(t)$ with $o(0)=o$, $v(t) \in \fu^-$ and $x(t) \in X$.
Expand $v(t) = v + t\,\dot{v} + O(t^2)$ and $x(t) = x + t\,\dot{x} + O(t^2)$.
where $\dot{x}$ (resp. $\dot{v}, \dot{o}$) the differential of $x(t)$ (resp. $v(t), o(t)$) at $t=0$.
Then $\dot{x} \in T_x X$ and  $\dot{v} \in \fu^-$. Then we have the following formula.
\begin{lem} \label{l.dotz}
$
\dot{o}  = \Ad(\exp(v))\bigl([\dot{v}, x] + \dot{x}\bigr).
$
\end{lem}
\begin{proof}

A direct computation shows that for any $g(t)\in \fg$, we have
\[
\left.\cfrac{d}{dt}\right|_{t=0} \left( g(t)x(t)g(t)^{-1} \right) = g(0)\left[g(0)^{-1}\dot{g}(0), x\right]g(0)^{-1} + g(0)\dot{x} g(0)^{-1}.
\]
Take $g(t) = \exp(v(t))$ and write $v(t) = v + t \dot{v} + t^2 \epsilon_t$.
As $\fu^-$ is abelian, $[v,t \dot{v} + t^2 \epsilon_t]=0$. Thus
\[
\exp(v(t)) = \exp(v)\exp(t \dot{v} + t^2 \epsilon_t) = \exp(v) + t \exp(v) \dot{v} + O(t^2).
\] 
This gives $g(0)=\exp(v)$ and $\dot{g}(0) = \exp(v) \dot{v}$, hence $g(0)^{-1}\dot{g}(0) =\dot{v}$, and thus the required equality follows.
\end{proof}

Now we define a regular function $f\colon U\rightarrow \bC$ as follows:
\[
f\colon U\longrightarrow \bC,\quad o=\Ad(\exp(v))x\longmapsto f(o)\coloneqq \kappa(x,v).
\]
Note that this is independent of the choice of $v$ (as the difference lies in $\fg_x$ while $\kappa(\fg_x, x)=0$).

\begin{lem}
    Let $\Ad(\exp(v))x=o\in U$ and let $[g,o]\in  [\fg,o] \cong T_o\cO$ be a tangent vector. Then 
\[
df_o([g,o]) =\kappa(g, x) +\cfrac{1}{2} \kappa([v,x], [v,g]).
\]
\end{lem}
\begin{proof}
 Choose a curve $o(t)=\Ad(\exp(v(t)))x(t)$ with $o(0)=o$ and $\dot{o}(0)=[g,o]$.  By Lemma \ref{l.dotz}, we have
\[
[g,o] = \Ad(\exp(v))\left([\dot{v}, x] + \dot{x}\right).
\]
Applying $\Ad(\exp(-v))$ to both sides yields
\[
[\Ad(\exp(-v))g, x] = [\dot{v}, x] + \dot{x}.
\]
Notice that the right hand side splits into a part in $\fl$ ($[\dot{v},x]$) and a part in $\fu^+$ ($\dot{x}$). Now decompose $\Ad(\exp(-v))g=u^- + l + u^+$ with $u^-\in\fu^-$, $l\in\fl$ and $u^+\in\fu$. Note that $[\fu^+,x]=0$ as $\fu^+$ is abelian. So one gets
\[
[\dot{v}, x] = [u^-, x] \quad \textup{and}\quad \dot{x} = [l, x].
\]
Note that $\dot{v} -u^- \in \fg_x$, so differentiating $f(o(t))=\kappa(x(t), v(t))$ at $t=0$ yields
\begin{align*}
    df_o([g,o]) = \kappa(\dot{x}, v) + \kappa(x, \dot{v})
 & = \kappa([l,x], v) + \kappa(x, u^-) \\
 & = \kappa(x, u^-) - \kappa(l, [v, x]).
\end{align*}
As $\kappa(x,\fl)=\kappa(x,\fu^+)=\kappa(u^-,[v,x])=\kappa(u^+,[v,x])=0$, we get
\[
\kappa(x, \Ad(\exp(-v))g) = \kappa(x, u^-)\quad \text{and}\quad \kappa(\Ad(\exp(-v))g, [v,x]) = \kappa(l, [v,x]),
\]
and hence
\[
df_o([g,o]) = \kappa\left(x,\; \Ad(\exp(-v))g\right) - \kappa\left(\Ad(\exp(-v))g,\; [v,x]\right).
\]
On the other hand, as $\textup{ad}_v^3=0$, we have 
\[
\Ad(\exp(-v))g = g -[v,g] + \cfrac{1}{2}[v,[v,g]].
\]
This yields
\begin{align*}
    \kappa\left(x,\; \Ad(\exp(-v))g\right) &= \kappa(x,g) - \kappa(x, [v,g]) + \cfrac{1}{2} \kappa(x, [v,[v,g]]) \\ 
    &= \kappa(x,g) + \kappa(g, [v,x]) - \cfrac{1}{2} \kappa([v, x], [v,g])
\end{align*}
Note that $[v,[v,g]] \in \fu^-$ for any $g\in \fg$, thus $\kappa([v,[v,g]], [v,x])=0$ and we obtain
\[
\kappa\left(\Ad(\exp(-v))g, [v,x]\right) = \kappa(g, [v,x]) - \kappa([v,g], [v,x]),
\]
from which the required equality follows immediately.
\end{proof}

\begin{lem} \label{l.Liouville1form}
    The pull-back of the Liouville $1$-form on $T^*X$ to $U$ is given by 
\[
\eta_o([g,o])= \kappa(g,[v,x]) - \kappa([v,x], [v, g])\quad  g\in \fg,\;\Ad(\exp(v))x=o\in U.
\]
\end{lem}
\begin{proof}
Write $g=g^+ + g^l + g^-$ and $o=x+[v,x]+\frac{1}{2}[v,[v,x]]$. As $[g^+, x]=0$, the projection of $[g,o]\in T_{o}\cO$ to $T_x X\cong [\fl,x]$ is given by 
\[
[g, o]^+ = [g^l, x] + [g^+, [v,x]]  =  [g^l, x] +[[g^+, v], x].
\]
Thus, the Liouville $1$-form on $U$ is given by

\begin{align*}
\eta_o([g,o])  = \langle [v,x], [g,o]^+\rangle & = \omega^{KKS}_o([g^l,x]+[[g^+,v],x],[v,x]) \\
& = \kappa(x, [g^l, v]) + \kappa(x, [[g^+, v], v]) \\
&= \kappa([v, x], g^l) - \kappa([v, x], [v, g^+]) \\
& = \kappa(g,[v,x]) - \kappa([v,x], [v, g]). 
\end{align*}
Here we first identify $o\in \cO$ to the point $(x,[v,x])\in E$ and then regard $[v,x]$ as a covector in $T^*_x X$ defined by \eqref{eq.E-to-Cotang}.
\end{proof}
\begin{proof}[Proof of Proposition \ref{p.symiso}]
Recall that the Kostant--Kirilov--Souriau symplectic form $\omega^{KKS}$ on $\cO$ is in fact exact, which is given by $\omega^{KKS} = d \beta$, where $\beta$ is the 1-form on $\cO$ defined as follows (cf.~\cite[Proposition 6.1]{Jia2021}):
\[
\beta_o([g,o]) \coloneqq \kappa(g, o) = \kappa(g,x) + \kappa(g,[v, x]) - \cfrac{1}{2} \kappa([v,x],[v,g])
\]
While by Lemma \ref{l.Liouville1form}, we have 
\[
\eta_o([g,o])= \kappa([v,x], g) - \kappa([v,x], [v, g]).
\]
It follows that $\eta + df = \beta$ and hence $d\eta = d\beta = \omega^{KKS}$, which implies that the isomorphism $U \cong T^*X$ is indeed symplectic.
\end{proof}

\section{Hamiltonian reduction of \texorpdfstring{$\overline{\bfO}_n\subset \mathfrak{so}_{2n+2}$}{On}}
\label{s.HamRed}

The aim of this section is to investigate a Hamiltonian reduction of the minimal nilpotent orbit closure $\overline{\bfO}_n$ in $\mathfrak{so}_{2n+2}$ with $n\geq 2$ under a specified $\bC^*$-action. Recall also that we denote by $\overline{\bO}_n\subset \mathfrak{sl}_{n}$ the minimal nilpotent orbit closure of type A.

\subsection{Setup}
Let us identify $\mathfrak{so}_{2n+2}$ to the space consisting of matrices of the following form:
\begin{equation}
\label{eq.Form-xi}
	\xi\coloneqq \begin{pmatrix}
		w       &  \bfu    &   0       &     \bfv      \\
		\bfx^t  &   A      & -\bfv^t   &      B       \\
		0       & \bfy     & -w        &      -\bfx    \\
		-\bfy^t & C        & -\bfu^t   &      -A^t    
	\end{pmatrix}
\end{equation}
where $w\in \bC$, $\bfu=(u_i),\bfv=(v_i),\bfx=(x_i),\bfy=(y_i)\in \bC^n$, $A=(a_{ij}),B=(b_{ij}), C=(c_{ij})\in
\operatorname{Mat}_{n\times n}(\bC)$ such that  
\[
C=-C^t\quad \text{and}\quad B=-B^t.
\]
 Given $\xi$, we will denote by $\bfr_i$ and $\bfl_i$ the $i$-th row vector and the $i$-th column vector, respectively. Let $\fp$ be the parabolic subalgebra of $\mathfrak{so}_{2n+2}$ consisting of $\xi$ with $\bfu=\bfv=0$. Let $\fu^+$ be the nilpotent radical of $\fp$ and $\fl$ its Levi part. Then we have a decomposition
\[
\mathfrak{so}_{2n+2} = \fu^+ \oplus \fz(\fl) \oplus [\fl,\fl] \oplus \fu^- \cong \fu^+ \oplus \bC \oplus \mathfrak{so}_{2n} \oplus \fu^-
\]
This decomposition can be explicitly described as follows: Each component consist of elements $\xi\in \mathfrak{so}_{2n+2}$ where only specific components in $\xi$ are non-zero:
\begin{itemize}
	\item $\fu^+$ corresponds to the components $\bfx$ and $\bfy$.
	
	\item $\fu^-$ corresponds to the components $\bfu$ and $\bfv$.
	
	\item $\fz(\fl)$ corresponds to the component $w$.
	
	\item $[\fl,\fl]$ corresponds to the components $A$, $B$ and $C$.
\end{itemize}
The minimal nilpotent orbit closure $\overline{\bfO}_n$ in $\mathfrak{so}_{2n+2}$ can be written as
\[
\overline{\bfO}_n\coloneqq \left\{\xi\in \mathfrak{so}_{2n+2} \mid \xi^2=0\,\text{and}\,\rank(\xi)\leq 2\right\}.
\]

\begin{rem}
	The Lie algebra $\mathfrak{so}_4\cong \mathfrak{sl}_2\times \mathfrak{sl}_2$ is not simple. By abuse of notation, we also denote by $\bfO_1$ the locally closed subset of $\mathfrak{so}_4$ consisting of $\xi\in \mathfrak{so}_4$ such that $\xi^2=0$ and $\rank(\xi)=2$. Then an easy computation shows that $\overline{\bfO}_1$ consists of two $\bC$-lines meeting at $0$.
\end{rem}

Let us consider the following $\bC^*$-subgroup of ${\rm SO}_{2n+2}$ (in the matrix form of \eqref{eq.Form-xi}) that induces a $\bC^*$-action on $\mathfrak{so}_{2n+2}$ via an adjoint action:
$$
{\rm diag}(1, \lambda, \cdots, \lambda, 1, \lambda^{-1}, \cdots, \lambda^{-1}).
$$
This $\bC^*$-action is Hamiltonian with respect to the Kostant--Kirillov--Saurian form on $\bfO_n$ and a direct computation shows that  the $\bC^*$-action on $\overline{\bfO}_n$ is given as follows:
\[
\lambda\cdot w= w, \quad\lambda\cdot \bfx=\lambda\bfx,\quad \lambda\cdot \bfy=\lambda^{-1}\bfy,\quad \lambda\cdot \bfu=\lambda^{-1}\bfu,\quad \lambda\cdot \bfv=\lambda\bfv
\]
and
\[
\lambda\cdot A = A,\quad \lambda\cdot B=\lambda^2 B,\quad\lambda\cdot C=\lambda^{-2}C,
\]
so that the blocks in \eqref{eq.Form-xi} correspond exactly to the following weight decomposition:
\begin{equation}
	\label{eq.Weight-decomp}
	\mathfrak{so}_{2n+2} = \fu^+_{-1}\oplus \fu^+_{1} \oplus \fz(\fl) \oplus [\fl,\fl]_{-2} \oplus [\fl,\fl]_0\oplus [\fl,\fl]_2 \oplus \fu^-_{-1}\oplus \fu^-_{1}.
\end{equation}
 Let $\mu\colon \overline{\bfO}_n\rightarrow \fc^*\cong \bC$ be the associated moment map \cite[\S\,2.1]{FuLiu2026}. Then a straightforward computation shows:
\begin{lem}
    The moment map $\mu$ sends $\xi$ to $\operatorname{Tr}(A)$.
\end{lem}

\subsection{Geometry of the shell $N$}
\label{ss.Shell}


Let $N\coloneqq \mu^{-1}(0)$ be the scheme-theoretic fiber of the moment map $\mu$, called the \emph{shell}. Denote by $M$ the Hamiltonian reduction $\overline{\bfO}_n/\!\!/\!\!/\bC^*:=N/\!/\bC^*$. We will give a detailed description of the geometry of the shell $N$ and also the $\bC^*$-orbits on $N$. Set $\bZ_2\coloneqq \bZ/2\bZ$. Let us introduce three locally closed subsets of $N$ as follows:
\begin{align*}
	N_{(e)} 
	  & \coloneqq \{ \xi \in N \mid ({\bf x} \neq 0 \text{ and } {\bf y} \neq 0) \text{ or } ({\bf u} \neq 0 \text{ and } {\bf v} \neq 0) \}.\\
	N_{(\bZ_2)}
	  & \coloneqq \left(N\cap \fl\right)\setminus [\fl,\fl]_0.\\
	N_{(\bC^*)}
	  & \coloneqq N\cap [\fl,\fl]_0.
\end{align*}

\begin{lem}
	\label{lem.Eq-N}
	The shell $N$ is an irreducible normal variety defined by $\operatorname{Tr}(A)=0$ as in \eqref{eq.Form-xi} such that $N_{(\bC^*)}=N_{\sing}$ is isomorphic to $\overline{\bO}_{n}\subset [\fl,\fl]_0\cong \mathfrak{sl}_{n}$.
\end{lem}

\begin{proof}
	As $\mu(\xi)=\operatorname{Tr}(A)$, the shell $N$ is defined by $\operatorname{Tr}(A)=0$. Note that $N_{\sing}$ consists of points in $N$ whose isotropy group is not finite \cite[Lemma 2.6]{FuLiu2026}. Thus an easy computation shows $N_{\sing} = N \cap [\fl,\fl]_0$. In particular, the projection mapping $\xi$ to $A$ yields an isomorphism $N_{\sing} \cong \overline{\bO}_n$. Hence, we have
	\[
	\dim N_{\sing} = 2n-2 = \dim N - (2n-1).
	\]
	On the other hand, since $\overline{\bfO}_n$ is Cohen-Macaulay and $n \geq 2$,  the shell $N$ is a normal variety \cite[Lemma 2.6]{FuLiu2026}. Finally, note that the projectivization $\bP(\bfO_n) \subset \bP(\mathfrak{so}_{2n+2})$ is a $(4n-3)$-dimensional rational homogeneous space of Picard number one with its minimal embedding, and $N$ is the affine cone over the irreducible hyperplane section $\bP(N)$ of $\bP(\bfO_{n})$. Hence it is indeed irreducible.
\end{proof}

The following result justifies the choice of our notation.

\begin{prop}
	\label{prop.Orbits-N}
	Let $\xi \in N$ be a point whose $\bC^*$-orbit $O_{\xi} =  \bC^* \cdot\xi$ is closed.
	\begin{enumerate}
		\item\label{i1.prop.Orbits-N} The isotropy subgroup $G_{\xi}$ is isomorphic to one of the following:
		\[
		\{e\},\quad \bZ_2,\quad \bC^*.
		\]
		
		\item\label{i2.prop.Orbits-N} $G_{\xi} \cong \bC^*$ if and only if $O_{\xi} \subset N_{(\bC^*)}$.
		
		\item\label{i3.prop.Orbits-N} $G_{\xi} \cong \bZ_2$ if and only if $O_{\xi} \subset N_{(\bZ_2)}$.
		
		\item\label{i4.prop.Orbits-N} $G_{\xi} \cong \{e\}$ if and only if $O_{\xi} \subset N_{(e)}$. 
	\end{enumerate}
\end{prop}

\begin{proof}
	The first statement is obvious. Now note that $G_{\xi}=\bC^*$ only if $\xi \in \fz(\fl)\oplus [\fl,\fl]_0$. On the other hand, as $\xi^2=0$, one gets 
	\[
	N\cap (\fz(\fl)\oplus [\fl,\fl]_0) = N\cap [\fl,\fl]_0,
	\]
	which proves the second statement. For the third statement, one observes that $G_{\xi}=\bZ_2$ only if $\xi \in \fl\setminus (\fz(\fl)\oplus [\fl,\fl]_0)$. Then the result follows as $N\cap \fl = N\cap [\fl,\fl]$

	For the last statement, let $\xi \in N\setminus N_{(e)}$ be a point with trivial isotropy subgroup, so that $\xi \in N\setminus \fl$, that is, one of $\bfx$, $\bfy$, $\bfu$ and $\bfv$ is non-zero. Without loss of generality, we may assume that $\bfx \neq 0$ so that $\bfy=0$ as $\xi\not\in N_{(e)}$, and the other cases are analogous. We will prove that $O_{\xi}$ is not closed in $N$, i.e., $C=0$ and $\bfu=0$. Now suppose that $x_i\not=0$. Then one observes that the vectors $\bfr_{i+1}$ and $\bfr_{n+2}$ (respectively $\bfl_1$ and $\bfl_{n+i+2}$) are linearly independent and non-zero. Recall that $\xi \in N$ if and only if the following conditions hold:
	\begin{equation}
		\label{eq.N}
		\xi^2 = 0,\quad \rank(\xi) \leq 2, \quad \operatorname{Tr}(A) = 0.
	\end{equation}
	The second condition means that every row $\bfr_j$ (resp. every column $\bfl_j$) of $\xi$ can be expressed as a linear combination of $\bfr_{i+1}$ and $\bfr_{n+2}$ (resp. $\bfl_1$ and $\bfl_{n+i+2}$). As a consequence, as $\bfy = 0$, for each $j \geq n+2$ we have $\bfr_j = \lambda_j \bfr_{n+2}$ for some $\lambda_j \in \bC$, which implies $C = 0$.  Similarly, for $1 \leq j \leq n$, one gets 
	\[
	a_{jj} = \frac{x_j}{x_i} a_{ij},
	\]
	and therefore the last condition of \eqref{eq.N} yields
	\[
	0 = x_i \operatorname{Tr}(A)  = \sum_{j=1}^n x_j a_{ij}.
	\]
	On the other hand, note that the first condition in \eqref{eq.N} can be reformulated as follows:
	\[
	\bfr_{i+1}\bfl_1=\bfr_{i+1}\bfl_{n+i+2}=\bfr_{n+2}\bfl_1=\bfr_{n+2}\bfl_{n+i+2}=0
	\]
	which in particular implies 
	\[
	0=\bfr_{i+1}\bfl_1=	wx_i+\sum_{j=1}^n x_j a_{ij}=0.
	\]
	Since $x_i \neq 0$, it follows that $w = 0$ and hence $\bfr_1 = \lambda_1 \bfr_{n+2}$ for some $\lambda_1 \in \bC$. This implies that $\bfu = 0$ as $\bfy=0$.
\end{proof}

One can easily see that every $\bC^*$-orbit contained in $N_{(\bC^*)}$ or $N_{(e)}$ is closed. For  the orbits in $N_{(\bZ_2)}$, we prove:

\begin{lem}
	\label{lem.Orbits-Z2}
	$\dim N_{(\bZ_2)}=4n-7$ and the following statements hold.
	\begin{enumerate}
		\item If $n\geq 3$, then $\overline{N}_{(\bZ_2)}=N\cap [\fl,\fl]$ is irreducible.
		
		\item If $n=2$ or $3$, then there are no closed $\bC^*$-orbits contained in $N_{(\bZ_2)}$.
		
		\item If $n\geq 4$, then a general $\bC^*$-orbit contained in $N_{(\bZ_2)}$ is closed in $N$.
	\end{enumerate}
\end{lem}

\begin{proof}
	We will identify $[\fl,\fl]$ with $\mathfrak{so}_{2n}$, which consists of matrices of the following form:
	\[
	\theta\coloneqq \begin{pmatrix}
		A &  B \\
		C & -A^t
	\end{pmatrix}\; \text{with}\;
	B=-B^t,\; C=-C^t.
	\]
	The natural map $\overline{\bfO}_n\cap \fl\rightarrow [\fl,\fl]\cong \mathfrak{so}_{2n}$, $\xi\mapsto \theta$, then yields an isomorphism 
	\[
	\overline{\bfO}_n\cap \fl \cong \overline{\bfO}_{n-1}.
	\]
	Under this isomorphism, the closed subset $N\cap \fl$ is defined by the following equations:
	\begin{equation}
		\label{eq.On-1}
			\rank(\theta)\leq 2,\quad \theta^2=0,\quad \operatorname{Tr}(A)=0.
	\end{equation}
	If $n\geq 3$, note that $\bP (\bfO_{n-1})\subset \bP([\fl,\fl])$ is a rational homogeneous space of Picard number one and with its minimal embedding such that $N\cap \fl=N\cap [\fl,\fl]$ is the affine cone of a hyperplane section of $\bP(\bfO_{{n-1}})$, so it is irreducible. Then the first statement follows because $N_{(\bZ_2)}$ is an open subset of $N\cap [\fl,\fl]$.

	For the last two statements, without loss of generality, we may assume that $B\not=0$ so that $b_{ij}=-b_{ji}\not=0$ for some $1\leq i\not=j\leq n$. Denote by $\bfm_i$ (resp. $\bfn_i$) the $i$-th row vector (resp. column vector) of $\theta$. Then the first condition in \eqref{eq.On-1} is equivalent to the fact that $\bfm_i$ and $\bfm_j$ (resp. $\bfn_{n+i}$ and $\bfn_{n+j}$) are linearly independent such that every row vector (resp. column vector) of $\theta$ can be written as a linear combination of them. In particular, the second equation in \eqref{eq.On-1} can be reformulated as follows:
	\[
	\bfm_i \bfn_{n+i} = \bfm_i \bfn_{n+j} = \bfm_j \bfn_{n+i} = \bfm_j \bfn_{n+j} = 0.
	\]
	Then an easy computation shows that the algebraic variety 
	\[
	U\coloneqq N\cap \fl\cap \{b_{ij}=-b_{ji}\not=0\}
	\]
	is isomorphic to the locally closed subvariety of
	\[
	\bC^{4n-3}\coloneqq \{(\bfm_i,\bfm_j)\in \bC^{4n}\mid b_{ii}=b_{jj}=0\;\,\text{and}\;\,b_{ij}=-b_{ji}\}
	\]
	defined by the following equations:
	\begin{equation}
		\begin{dcases}
			\bfm_i \bfn_{n+i} = -2 \sum_{s=1}^n a_{is} b_{is} = 0 \\
			\bfm_i \bfn_{n+j} = \bfm_j \bfn_{n+i} = \sum_{s=1}^n (-a_{is} b_{js} - a_{js} b_{is}) = 0 \\
			\bfm_{j} \bfn_{n+j} = -2 \sum_{s=1}^n a_{js} b_{js} = 0 \\
			b_{ij}\operatorname{Tr}(A) = \sum_{s=1}^n (-a_{is}b_{js} + a_{js} b_{is}) = 0 \\
			b_{ij}=-b_{ji}\not=0
		\end{dcases}
	\end{equation}
	Then one can easily derive that $U$ is actually isomorphic to the open subset of 
	\[
	\bC^{4n-7}\coloneqq\{(\widehat{\bfm}_i,\widehat{\bfm}_j)\}
	\]
	defined by $b_{ij}\not=0$, where
	\[
	\widehat{\bfm}_i\coloneqq (a_{i1},\dots,\widehat{a}_{ii},\dots,\widehat{a}_{ij},\dots,a_{in},b_{i1},\dots,\widehat{b}_{ii},\dots,b_{ij},\dots,b_{in})\in \bC^{2n-3}
	\]
	and
	\[
	\widehat{\bfm}_j\coloneqq (a_{j1},\dots,\widehat{a}_{ji},\dots,\widehat{a}_{jj},\dots,a_{jn},b_{j1},\dots,\widehat{b}_{ji},\dots,\widehat{b}_{jj},\dots,b_{jn}) \in \bC^{2n-4}.
	\]
	Since the rows of $\theta$ can be written as a linear combination of $\bfm_i$ and $\bfm_j$, for any $1\leq k,l\leq n$, a straightforward computation then yields 
	\[
	c_{kl} = \frac{a_{ik}a_{jl} - a_{il}a_{jk}}{b_{ij}}.
	\]
	For $n=2$ or $3$, one can then derive that $c_{kl}=0$ for any $k,l$ and consequently $O_{\xi}$ is not closed. Nevertheless, for $n\geq 4$, one can choose $k,l$ different from $i,j$ so that $c_{kl}\not=0$ for any general $\widehat{\bfm}_i$ and $\widehat{\bfm}_j$, hence a general $\bC^*$-orbit is closed.	 
\end{proof}

\begin{rem}
	\label{rm.barNZ2}
	\begin{enumerate}
		\item By the above arguments, for $n=2$, the inclusion $\overline{N}_{(\bZ_2)}\subset N\cap \fl=N\cap [\fl,\fl]$ is strict. Indeed, the variety $N\cap [\fl,\fl]$ contains three irreducible components, two of which are isomorphic to $\bC$ and are contained in $\overline{N}_{(\bZ_2)}$, while the third is $N_{(\bC^*)}$, which is isomorphic to the $2$-dimensional quadric cone $\overline{\bO}_2\subset \mathfrak{sl}_2$.
		
		\item The same computation as in the proof of Lemma \ref{lem.Orbits-Z2} will be applied several times in the sequel. To simplify the exposition, in the latter we will only state the consequence of the computation, leaving the details to the appendix.
	\end{enumerate}
\end{rem}

\subsection{Geometry of $\overline{\bfO}_n/\!\!/\!\!/\bC^*$}

Recall that the Hamiltonian reduction $M\coloneqq \overline{\bfO}_n/\!\!/\!\!/\bC^*$ is defined as follows \cite[\S\,3.3]{FuLiu2026}:
\[
M = \overline{\bfO}_n/\!\!/\!\!/\bC^* \coloneqq N/\!/\bC^*.
\]
Let $\pi_N\colon N\rightarrow M$ be the quotient map. Then $M$ parameterizes the closed orbits contained in the shell $N$, and it can be stratified into a locally closed subset as follows \cite[\S\,3.2]{FuLiu2026}:
\begin{equation}
\label{eq.Iso-Dec-M}
    M = M_{(e)} \sqcup M_{(\bZ_2)} \sqcup M_{(\bC^*)}
\end{equation}
where $M_{(H)}$ parameterizes the closed orbits whose stabilizer is conjugate to $H$, that is, $M_{(H)}=\pi_N(N_{(H)})$. Thanks to Proposition \ref{prop.Orbits-N}, Lemmas \ref{lem.Eq-N} and \ref{lem.Orbits-Z2}, we have the following:
	\begin{enumerate}
		\item The stratum $M_{(e)}$ is an open subset of dimension $4n-4$.
		
		\item The stratum $M_{(\bZ_2)}$ is empty if and only if $n\leq 3$, and if $n\geq 4$, then $\overline{M}_{(\bZ_2)}=M\setminus M_{(e)}$ is irreducible and of dimension $4n-8$.
		
		\item The stratum $M_{(\bC^*)}$ is an irreducible closed subset of dimension $2n-2$, which is isomorphic to $\overline{\bO}_{n}\subset \mathfrak{so}_{2n}$.
	\end{enumerate}
The natural scaling on $\overline{\bfO}_n$ induces a cone structure on $M$ with vertex $o=\pi_N(0)$. In the sequel of this section, we aim to give a detailed description of the singularities of $M$ away from its vertex $o$. Firstly, we observe:

\begin{prop}
	If $n\geq 3$, the Hamiltonian reduction $M$ is a symplectic variety with terminal singularities.
\end{prop}

\begin{proof}
	The open stratum $M_{(e)}$ is non-singular and carries a symplectic form by the Hamiltonian slice theorem \cite[Theorem 3.2 and Lemma 3.6]{FuLiu2026}. Moreover, we also have
	\[
	\dim M\setminus M_{(e)} \leq \max\{2n-2,4n-8\}.
	\] 
	So $\codim M\setminus M_{(e)}\geq 4$ if $n\geq 3$, which implies that $M$ is a symplectic variety by \cite{Flenner1988} and with terminal singularity by \cite{Namikawa2001a}.
\end{proof}

\begin{rem}
	With a different method, we prove in Corollary \ref{cor.n=2-Msymplectic} that $M$ is also symplectic for $n=2$.
\end{rem}

Let $T=(\bC^*)^r$ be a torus acting on $\bC^n$ with weights $(\bfa_1,\dots,\bfa_n)$. Recall that the associated \emph{hypertoric variety} is the Hamiltonian reduction $T^*\bC^n/\!\!/\!\!/T$. In this case, the moment map can be written explicitly as follows:
	\[
	T^*\bC^n = \bC^n\times (\bC^n)^* \longrightarrow \ft\cong \bC^r,\quad (\bfx,\bfy)\longmapsto \sum_{i=1}^n \bfa_i x_i y_i.
	\]
    Given $n\geq 2$, we denote by $\cY_n$ the hypertoric variety associated to the $\bC^*$-action on $\bC^n$ with weights $(1,-1,(-2)^{n-2})$.

\begin{prop}
	\label{prop.Sing-M(C*)}
	Let $o\not=z\in M_{(\bC^*)}$ be a point. Then the germ of $M$ at $z$ is analytically isomorphic to the germ of $\cY_n\times \bC^{2n-2}$ at the origin.
\end{prop}

\begin{proof}
	Let $0\not=\xi\in N_{(\bC^*)}$ be the unique $\bC^*$-fixed point such that $\pi_N(\xi)=z$. Then $0\not=\xi\in[\fl,\fl]_0$. In particular, we may assume that $a_{ij}\not=0$ in $\xi$ for some $1\leq i\not=j\leq n$. Denote by $U$ the open subset $\overline{\bfO}_n\cap\{a_{ij}\not=0\}$. According to \S\,\ref{ss.prop.Sing-M(C*)}, the variety $U$ is isomorphic to the open subset of $\bC^{4n-2}\coloneqq \left\{(\widehat{\bfr}_{i+1},\widehat{\bfr}_{n+j+2})\right\}$
	defined by $a_{ij}\not=0$, where 
	\[
	\hat{\bfr}_{i+1}\coloneq (x_i,a_{i1},\dots,\hat{a}_{ii},\dots,a_{in},v_i,b_{i1},\dots,\hat{b}_{ii},\dots,\hat{b}_{ij},\dots,b_{in}) \in \bC^{2n-1}
	\]
	and
	\[
	\hat{\bfr}_{n+j+2} = (y_j,c_{j1},\dots,\hat{c}_{ji},\dots,\hat{c}_{jj},c_{jn},u_j,a_{1j},\dots,\hat{a}_{ij},\dots,a_{nj})\in \bC^{2n-1}.
	\]
    Moreover, this isomorphism is $\bC^*$-equivariant with respect to the natural $\bC^*$-action on $\bC^{4n-2}$, and the equation $-a_{ij}\operatorname{Tr}(A)=0$ can be rewritten with these new coordinates as follows:
	\begin{equation}
		\label{eq.Trace-Moment-map-0}
		x_i u_i - v_i y_i - 2\sum_{k=1}^n b_{ik} c_{jk} = 0.
	\end{equation}
    In particular, we can identify $\bC^{4n-2}$ with $T^*\bC^{2n-1}$ such that the $\bC^*$-action on it is induced by a $\bC^*$-action on $\bC^{2n-1}$ with weights 
    \[
    (1,-1,(-2)^{n-2},0^{n-1}),
    \]
    and the equation \eqref{eq.Trace-Moment-map-0} is given by the associated moment map. Then the associated hypertoric variety is isomorphic to $\cY_n\times \bC^{2n-2}$ and the Hamiltonian slice theorem \cite[Theorem 3.2]{FuLiu2026} implies that the germ of $M$ at $z$ is analytically isomorphic to the germ of $\cY_n\times \bC^{2n-2}$ at some point $(0,z')$ and hence at the origin.
\end{proof}

\begin{prop}
\label{prop.Sing-MZ2}
	For $n\geq 4$ and $z\in M_{(\bZ_2)}$, the germ of $M$ at $z$ is analytically isomorphic to the germ of $\bC^{4n-8}\times \bC^4/\langle \pm 1\rangle$ at the origin. In particular, the variety $M_{(\bZ_2)}$ is smooth.
\end{prop}

\begin{proof}
	Let $\xi\in N_{(\bZ_2)}$ with a closed $\bC^*$-orbit such that $\pi_N(\xi)=z$. Then one of $B$ and $C$ in $\xi$ is non-zero. Without loss of generality, we may assume that $B\not=0$ such that $b_{ij}=-b_{ji}\not=0$ for some $1\leq i\not=j\leq n$. Denote by $U_N$ the open subset $N\cap \{b_{ij}=-b_{ji}\not=0\}$ of $N$. According to \S\,\ref{ss.prop.Sing-MZ2}, the variety $U_N$ isomorphic to the open subset of 
	\[
	\bC^{4n-3}\coloneqq \left\{(\hat{\bfr}_{i+1},\hat{\bfr}_{j+1})\in \bC^{2n-1}\times \bC^{2n-1} \mid b_{ij}=-b_{ji}\right\},
	\]
	defined by $b_{ij}=-b_{ji}\not=0$, where 
	\[
	\hat{\bfr}_{i+1}\coloneqq (x_i,a_{i1},\dots,\hat{a}_{ii},\dots,\hat{a}_{ij},\dots a_{in}, -v_i,b_{i1},\dots,\hat{b}_{ii},\dots,b_{in})\in \bC^{2n-1}
	\]
	and 
	\[
	\hat{\bfr}_{j+1}\coloneqq (x_j,a_{j1},\dots,\hat{a}_{ji},\dots,\hat{a}_{jj},\dots,a_{jn},-v_j,b_{j1}\dots,\hat{b}_{jj},\dots,b_{jn})\in \bC^{2n-1}.
	\]
	Moreover, this isomorphism is $\bC^*$-equivariant with respect to the natural $\bC^*$-action on $\bC^{4n-3}$. Consider the natural action of the isotropy subgroup $G_{\xi}=\bZ_2$ on $T_\xi U$. Then one easily checks that the subspace $L$ of $\bC^{4n-2}$ defined by $b_{ij}=-b_{ji}=0$ is a $G_{\xi}$-stable complement of $T_{\xi} O_{\xi}$. Hence Luna's slice theorem (see \cite[Theorem 5.3]{BrionKraftSchwarz2025}) shows that the germ of $M$ at $z$ is analytically isomorphic to the germ of the quotient $L/G_{\xi}\cong \bC^{4n-8}\times \bC^4/\langle \pm 1\rangle$ at the origin, and the germ of $M_{(\bZ_2)}$ at $z$ is analytically isomorphic to $\bC^{4n-8}\times \{0\}$, which is smooth.
\end{proof}

\subsection{$\bQ$-factorial terminalizations} 
\label{ss.Q-fact-term}

Denote by $R$ the ring $\sO(N)$ of regular functions  on $N$ and write $R = \oplus_{d\in \bZ} R_d$ the weight decomposition with respect to the $\bC^*$-action. Then we can construct some $\bQ$-factorial terminalizations of $M$ via the variations of GIT. More precisely, we define 
\[
M^+\coloneqq \Proj\left(\bigoplus_{d\geq 0} R_d\right) \quad \textup{and} \quad M^-\coloneqq \Proj\left(\bigoplus_{d\geq 0} R_{-d}\right).
\]
Then there exist natural projective morphisms
\[
\mu_+\colon M^+\longrightarrow M\quad \textup{and}\quad \mu_-\colon M^-\longrightarrow M.
\]
By symmetry, we treat only $M^+$ in detail; the case of $M^-$ follows analogously. Recall that a point $\xi\in N$ is called \emph{semi-stable} (in the sense of GIT) if there exist an integer $d> 0$ and $f\in R_d$ with $f(\xi)\not=0$. It is called \emph{stable} if it is semi-stable with finite stabilizer, and the $\bC^*$-orbits in $N\setminus\{f=0\}$ are closed. Then one can easily derive that the semi-stable locus $N^{\textup{ss}}$ coincides with the stable locus $N^{\textup{s}}$ and is given by the open subset
\[
N\setminus \left(N\cap (\mathfrak{so_{2n+2}})_{\leq 0}
\right),
\]
where $(\mathfrak{so}_{2n+2})_{\leq 0}$ corresponds to the subspace of $\mathfrak{so}_{2n+2}$ with non-positive weights. Hence, the semi-stable locus $N^{\textup{ss}}$ is contained in $N_{\reg}$. Let us denote by $[\xi]$ the image of $\xi \in N^{\textup{ss}}$ in $M^+$. Then Luna's slice theorem implies that $M^+$ is smooth at $[\xi]$ if $\xi \in N^{\textup{ss}}\setminus [\fl,\fl]_2$. Moreover, for a point $\xi\in N^{\textup{ss}}\cap [\fl,\fl]_2$, by the proof of Proposition \ref{prop.Sing-MZ2}, Luna's slice theorem then implies that the germ of $M^+$ at $[\xi]$ is analytically isomorphic to $\bC^4/\langle \pm 1\rangle \times \bC^{4n-8}$. Consequently, we obtain:

\begin{lem}
	The varieties $M^+$ and $M^-$ are singular with  $\bQ$-factorial terminal singularities.
\end{lem}

For a point $\xi \in N^{\textup{ss}}$, the image $\mu_+([\xi])$ is equal to $[O_{\xi}]\in M$ if the orbit $O_{\xi}=G\cdot \xi$ is closed in $N$, and is equal to the limit
\[
\left[\lim_{\lambda\to 0} \lambda \cdot \xi\right] \in M_{(\bC^*)}
\]
if $O_{\xi}$ is not closed in $N$. So the morphism $\mu_+$ is a birational projective morphism, which is an isomorphism outside $M_{(\bC^*)}$.

\begin{lem}
	\label{lem.fibers-mu+}
	Let $z\in M_{(\bC^*)}$ be a point with a representative $\xi_z\in N_{(\bC^*)}$.
	\begin{enumerate}
		\item\label{i1.Fiber-mu_+} If $z\not=o$, then the fiber $\mu_+^{-1}(z)$ is isomorphic to the weighted projective space $\bP(1^2,2^{n-2})$.
		
		\item\label{i2.Fiber-mu_+} The central fiber $\mu_+^{-1}(o)$ consists of two irreducible components with dimension $2n-2$.
	\end{enumerate}
	In particular, the morphism $\mu_+$ is a small contraction if $n\geq 3$ and is a divisorial contraction with an irreducible exceptional divisor if $n=2$.
\end{lem}

\begin{proof}
	For \ref{i1.Fiber-mu_+}, without loss of generality, we assume that $\xi_z\in N_{(\bC^*)}$ with $a_{ij}\not=0$. As in the proof of Proposition \ref{prop.Sing-M(C*)}, one can show that the subvariety of $N^{\textup{ss}}$ consisting of points $\xi \in N^{\textup{ss}}$ such that $\mu_+([\xi])=z$ is isomorphic to the following vector space with the origin removed:
	\[
	(x_i,v_i,b_{i1},\dots,\hat{b}_{ii},\dots,\hat{b}_{ij},\dots,b_{in})\in \bC^{2n},
	\]
	where $\bC^*$ acts with weights $(1^2,2^{n-2})$, as required.
	
	For \ref{i2.Fiber-mu_+}, let us denote by $U$ the subvariety consisting of points $\xi \in N^{\textup{ss}}$ such that $[\xi]\in \mu^{-1}_+(o)$. Then a point $\xi \in N^{\textup{ss}}$ is contained in $U$ if and only if it is of the following form:
	\begin{equation}
		\label{eq.Xi-centralfiber}
		\xi\coloneqq \begin{pmatrix}
			0       &  0    &   0       &     \bfv      \\
			\bfx^t  &   0      & -\bfv^t   &      B       \\
			0       & 0     & 0        &      -\bfx    \\
			0 & 0        & 0   &      0    
		\end{pmatrix}\not=0,\quad \xi^2=0\quad \textup{and}\quad \rank(\xi)\leq 2.
	\end{equation}
	According to \S\,\ref{ss.fibers-mu+}, the variety $U$ consists of  two irreducible components with dimension $2n-1$, which then correspond to the two irreducible components of $\mu_+^{-1}(o)$ with dimension $2n-2$.
\end{proof}

\begin{cor}
	\label{cor.n=2-Msymplectic}
	If $n=2$, the Hamiltonian reduction $M$ is symplectic.
\end{cor}

\begin{proof}
	Since $M_{\reg}=M_{(e)}$ carries a symplectic form, it remains to prove that $M$ has rational singularities by \cite[Theorem 6]{Namikawa2001}. Consider the $\bQ$-factorial terminalization $\mu_+\colon M^+\rightarrow M$. Note that $\cY_2$ is exactly the $A_1$-singularity and $\mu_+$ is the minimal resolution of $M$ along $M_{(\bC^*)}\setminus\{o\}$. In particular, the morphism $\mu^+$ is crepant, i.e. $\mu_+^*K_{M}=K_{M^+}$. since $M^+$ has terminal singularities, the variety $M$ has canonical singularities and hence rational singularities.
\end{proof}

\subsection{Conclusion}

We summarize the results obtained in this section in the following theorem.

\begin{thm}
\label{thm.M-Geometry}
    Let $M=M_{(e)}\sqcup M_{(\bZ_2)}\sqcup M_{(\bC^*)}$ be the isotropy type decomposition in \eqref{eq.Iso-Dec-M}. Then the following assertions hold. 
    \begin{enumerate}   
        \item The variety $M$ is a symplectic variety.
        
        \item The variety $M$ has only terminal singularities for $n\geq 3$.
    
       \item $M_{\reg} = M_{(e)}$.
       
       \item  $M_{\sing}=M_{(\bC^*)}$ for $2\leq n\leq 3$, and  $M_{\sing}=\overline{M}_{(\bZ_2)}$ for $n\geq 4$.
        
        \item If $n\geq 2$, the subvariety $M_{(\bC^*)}$ is isomorphic to $\overline{\bO}_n\subset \mathfrak{sl}_n$ and the germ of $M$ along a point $o\not= z\in M_{(\bC^*)}$ is analytically isomorphic to 
        \[
        \cY_n\times \bC^{2n-2}.
        \]

        \item \label{i.sing} If $n\geq 4$, the subvariety $M_{(\bZ_2)}$ is non-singular and irreducible such that the germ of $M$ along a point $z\in M_{(\bZ_2)}$ is analytically isomorphic to 
        \[
        \bC^4/\langle \pm 1\rangle\times \bC^{4n-8}.
        \]
        
        \item The variety $M$ does not admit symplectic resolutions.
    \end{enumerate}
\end{thm}

\begin{proof}
Note that the $\bQ$-factorial terminalizaton $M^+\rightarrow M$ is not smooth, so $M$ does not admit any symplectic resolutions by \cite[Corollary 31]{Namikawa2008a} and the last statement follows.
\end{proof}

\begin{rem}
For $n \geq 4$,  \ref{i.sing} implies that $M$ has no symplectic resolution, because  $\bC^4/\langle \pm 1\rangle$ has none. For $n=3$,  the non-existence of a symplectic resolution can also be deduced from that of $\mathcal{Y}_3$, while for $n=2$, we must use \cite[Corollary 31]{Namikawa2008a}.
\end{rem}

\section{Proofs of Theorems \ref{t.TOn-symp-res} and \ref{thm.main-thm}} 

Throughout this section, we retain the notation from \S\,\ref{s.HamRed}.	In particular, we denote by $\bfO_n$ and $\bO_n$ the minimal nilpotent orbit in $\mathfrak{so}_{2n+2}$ and $\mathfrak{sl}_n$, respectively, where $n\geq 2$. 

\subsection{A quotient construction}

Recall that if we identify $\mathfrak{sl}_n$ with the space consisting of traceless $n\times n$-matrices, then a $n\times n$-matrix $\theta$ is contained in $\overline{\bO}_n$ if and only if $\rank(\theta)\leq 1$ and $\theta^2=0$. Elements $\bfx=(x_i)\in \bC^n$ will be written as row vectors. Let $Y$ be the quadric cone in $\bC^{2n}$; that is,
	\[
	Y=\{(\bfx,\bfy)\in \bC^n\times \bC^n \mid \bfx \bfy^t =0 \}\subset \bC^n\times \bC^n.
	\]
	Let $\bC^*$ act on $Y$ by $\lambda\cdot (\bfx,\bfy)=(\lambda\bfx,\lambda^{-1}\bfy)$. Then we have
	\begin{equation}
		\label{eq.Quot-Y}
		q_Y\colon Y\longrightarrow Y/\!/\bC^*=\overline{\bO}_n, \quad (\bfx,\bfy)\longmapsto \theta=(x_iy_j)_{1\leq i,j\leq n}.
	\end{equation}
Let $\fp$ be the maximal parabolic subalgebra of $\mathfrak{so}_{2n+2}$ corresponding to the first simple root (in the order of Bourbaki). Write $\fp=\fl\oplus \fu^+$, where $\fu^+$ is the nilpotent radical and $\fl$ is the Levi part. Then we have	\begin{equation}
		\label{eq.dec-so2n+2}
		\mathfrak{so}_{2n+2}=\fu^+\oplus \fl \oplus \fu^-,
	\end{equation}
	where $\fu^-$ is the nilpotent radical of $\fp^-$. Applying Theorem \ref{thm.Quadric} to $(\mathfrak{so}_{2n+2},\fp_1,\bfO_n)$ yields a diagram as follows (see also \cite[Table 2]{FuLiu2025}):
    \[
    \begin{tikzcd}
        \overline{\bfO}_n \arrow[r,"{\pi_+}"] \arrow[d,"{\pi_-}" left]
            & Y^+\subset \fu^+ \\
        Y^-\subset \fu^-
            &
    \end{tikzcd}
    \]
    where 
    \begin{enumerate}
        \item the affine varieties $Y^+\subset \fu^+$ and $Y^-\subset \fu^-$ are isomorphic to $Y\subset \bC^{2n}$ in a natural way such that $Y_{\reg}\cong \bfO_n\cap \fu^+\cong \bfO_n\cap \fu^-$, and 
        \item the open subset $U^+\coloneqq \pi^{-1}_+(Y^+_{\reg})$ (resp. $U^-\coloneqq \pi_-^{-1}(Y^-_{\reg})$) is symplectically isomorphic to $T^*Y^+_{\reg}$ (resp. $T^*Y^-_{\reg}$).
    \end{enumerate}
    The $\bC^*$-action on $Y$ then induces a Hamiltonian $\bC^*$-action on $U^+$ (and also $U^-$), which coincides with the $\bC^*$-action on $\overline{\bfO}_n$ introduced in \S\,\ref{s.HamRed} (see, for instance, \cite[\S\,1]{LevasseurStafford1999}). Let $N\subset \overline{\bfO}_n$ be its shell (cf.~\S\,\ref{ss.Shell}).

\subsection{Fibers of $\pi_+$ in $N$}
     
     Under the isomorphism $U^+\cong T^*Y^+_{\reg}$, for any point $0\not=y\in Y^+$, one easily derives that the fiber of $\pi_+|_N\colon N\rightarrow Y^+$ is isomorphic to a codimension one subspace of $T_y Y^+\cong \bC^{2n-1}$. Denote by $F_+$ the fiber of $\pi_+$ over the vertex $0\in Y^+$. Let $G_+\coloneqq N\cap F_+$. Denote by $p_+$ the restricted map 
     \[
     \pi_-|_{G_+}\colon G_+\longrightarrow Y^-.
     \]
     
     \begin{lem}
     	\label{lem.Fibers-center-fiber}
     	\begin{enumerate}
     		\item The fiber of $p_+$ over $(\bfu,\bfv)\not=0$ with $\bfu=0$ or $\bfv=0$ is isomorphic to a quadric hypersurface in $\bC^{2n-2}$ with one isolated singular point.
     		
     		\item If $n\geq 3$, the fiber of $p_+$ over $(\bfu,\bfv)$ with $\bfu\not=0$ and $\bfv\not=0$ is isomorphic to a quadric hypersurface in $\bC^{2n-3}$ with one dimensional singular locus.
     		
     		\item If $n=2$, the fiber of $p_+$ over $(\bfu,\bfv)$ with $\bfu\not=0$ and $\bfv\not=0$ is isomorphic to $\bC$.
     	\end{enumerate}
     \end{lem}
     
     \begin{proof}
     	Let $0\not=y=(\bfu,\bfv)\in Y^-$ be a non-zero point.  Without loss of generality, we may assume that $\bfu\not=0$ with $u_i\not=0$ for some $1\leq i\leq n$. According to \S\,\ref{ss.lem.Fibers-center-fiber}, the fiber $F_y$ of $p_+$ is the closed subset of 
     	\[
     	\bC^{2n-1}\ni (c_{i1},\dots,\hat{c}_{ii},\dots,c_{in},a_{1i},\dots,a_{ni})
     	\]
     	defined by the following equations:
     	\begin{equation}
        \label{eq.Fy}
     		\sum_{k=1}^n u_k a_{ki}  = \sum_{k=1}^n v_k c_{ik} = \sum_{k=1}^n a_{ki} c_{ik} = 0,
     	\end{equation}
     	where $c_{ii}=0$. As $\bfu\bfv^t=0$, the required result follows by an easy computation.
     \end{proof}
     
     Let $Y^+_{\bfx}$ be the linear section $Y^+\cap \{\bfy=0\}$ of $Y^+$ and let $N_{\bfx}\coloneq \pi_+^{-1}(Y^+_{\bfx})\cap N$. We also define the varieties $Y^+_{\bfy}$, $Y^-_{\bfu}$, $Y^-_{\bfv}$ and $N_{\bfy}$ in a similar way.
     
     \begin{cor}
     	\label{cor.dim-G+}
     	The dimension of $G_+$ is $4n-5$ for $n\geq 3$, and $4$ for $n=2$.  
     \end{cor}
 
 \begin{proof}
 	One observes
 	\[
 	\dim G_+ = \max \left(\dim p_+^{-1}\left(Y^-\setminus (Y^-_{\bfu}\cup Y^-_{\bfv})\right), \dim p_+^{-1}\left((Y^-_{\bfu}\cup Y^-_{\bfv})\setminus \{0\}\right), \dim p_+^{-1}(0)\right).
 	\]
 	Moreover, note that $p_+^{-1}(0)$ is actually isomorphic to $N\cap \fl = N\cap [\fl,\fl]$. Thus one derives from Remark \ref{rm.barNZ2}  and Lemma \ref{lem.Fibers-center-fiber} that
 	\[
 	\dim G_+ = \max\{d(n),3n-3,d'(n)\},
 	\]
 	where
 	\[
 	d(n)=
 	\begin{dcases}
 		4n-5 & \text{if}\; n\geq 3 \\
 		4    & \text{if}\; n=2
 	\end{dcases}
 	\quad \text{and}\quad
 	d'(n)=
 	\begin{dcases}
 		4n-7 & \text{if}\; n\geq 3\\
 		2    & \text{if}\; n=2.
 	\end{dcases}
 	\]
 	This yields that $\dim G_+ = d(n)$.
 \end{proof}

\subsection{Proofs of Theorems \ref{t.TOn-symp-res} and \ref{thm.main-thm}}

Now we are in the position to finish the proof of our main theorems. 

\begin{proof}[Proof of Theorem \ref{thm.main-thm}]
	Denote by $Y^+_{\diamond}$ the open subset $Y^+\setminus (Y^+_{\bfx}\cup Y^+_{\bfy})$. Then the restricted map
	\[
	q\coloneq q_{Y^+}|_{Y^+_{\diamond}} \coloneq Y^+_{\diamond}\longrightarrow \bO_n
	\]
	is a surjective geometric quotient, where $q_{Y^+}\colon Y^+\rightarrow \overline{\bO}_n$ is the quotient map induced by $q_Y\colon Y\rightarrow \overline{\bO}_n$. The cotangent map $(dq)^*$ induces a map
	\[
	\alpha\colon q^*(T^*\bO_n) \xlongrightarrow{(dq)^*} T^*(Y^+_{\diamond}) \xlongrightarrow{\cong} \pi_+^{-1}(Y^+_{\diamond})
	\]
	such that $\operatorname{Im}(\alpha)=N\cap \pi_+^{-1}(Y^+_{\diamond})\eqqcolon N^+_{\diamond}$ and the induced morphism 
    \[
    q^*(T^*(\bO_n))\longrightarrow N^+_{\diamond}
    \]
    is an isomorphism. As $n\geq 3$, by Corollary \ref{cor.dim-G+}, we have
	\begin{align*}
		\dim N\setminus N^+_{\diamond} 
		         &  = \max\left\{\dim G_+, \dim \pi_+^{-1}\left((Y^+_{\bfx}\cup Y^+_{\bfy})\setminus\{0\}\right)\right\} \\
		         &  =  \max\{4n-5,3n-2\}\\
		         &  = 4n-5.
	\end{align*}
	It follows that $N\setminus N^+_{\diamond}$ has codimension two in $N$. Note that $\pi_N\colon N\rightarrow M\coloneqq N/\!/\bC^*$ is a geometric quotient over $N_{\diamond}^+$, one gets 
	\[
	\pi_N^{-1}\left(\pi_N(N^+_{\diamond})\right) = N^+_{\diamond}.
	\]
	Therefore $\pi_N(N\setminus N^+_{\diamond})$ has codimension at least two in $M$, and hence the affine closure of $\pi_N(N^+_{\diamond})$ is isomorphic to $M$. In particular, the isomorphism $q^*(T^*\bO_n)\rightarrow N^+_{\diamond}$ yields a natural isomorphism $T^*\bO_n\rightarrow \pi_N^+(N^+_{\diamond})$, which then implies that $\overline{T^*\bO_n}^{\aff}$ is isomorphic to $M$.
\end{proof}

\begin{rem}
\label{rem.n=2}
	Theorem \ref{thm.main-thm} does not hold for $n=2$. Indeed, if $n=2$, then the minimal nilpotent orbit closure $\overline{\bO}_2$ is isomorphic to $\bC^2/\langle\pm 1\rangle$ so that $\overline{T^*\bO_{2}}^{\aff}$ is isomorphic to $\bC^4/\langle \pm 1\rangle$. Nevertheless, by Proposition \ref{prop.Sing-M(C*)}, the singular locus of $M$ has codimension two, so $M$ is not isomorphic to $\overline{T^*\bO_2}^{\aff}$. On the other hand, one can also derive from Corollary \ref{cor.dim-G+} that $N\setminus N_{\diamond}^+$ has three irreducible components of codimension one, which provides a geometric interpretation of this fact (cf.~\cite[\S\,3]{LevasseurStafford1999}).
\end{rem}

\begin{proof}[Proof of Theorem \ref{t.TOn-symp-res}]
    By Theorems \ref{thm.M-Geometry} and \ref{thm.main-thm}, it remains to consider the case $n=2$. Indeed, as explained in Remark \ref{rem.n=2}, the variety $\overline{T^*\bO_2}^{\aff}$ is isomorphic to $\bC^4/\langle \pm 1\rangle$, which is $\bQ$-factorial and terminal. Hence it does not admit any symplectic resultions by \cite[Corollary 1.3]{Fu03}.
\end{proof}

As the second application of Theorem \ref{thm.main-thm}, we slightly strengthen \cite[Theorem 0.1]{LevasseurStafford1999}. We refer the reader to \cite[\S\,5]{FuLiu2026} for the related basic definitions.

\begin{cor}
	Let $\pi_Y\colon  Y\rightarrow \overline{\bO}_n$ be the quotient map in \eqref{eq.Quot-Y}. If $n\geq 3$, the push-forward map
	\[
	(\pi_Y)_*\colon \cD(Y)\longrightarrow \cD(\overline{\bO}_n)
	\]
	is graded surjective and the graded sympol map
	\[
	\overline{\sigma}\colon \gr \cD(\overline{\bO}_n) \longrightarrow \cO(T^*\bO_n)
	\] 
	is surjective.
\end{cor}

\begin{proof}
    By \cite[Theorem 4.2 and Remark 4.3]{FuLiu2025}, the natural graded map 
    \[
    \overline{\sigma}_Y\colon \gr \cD(Y)\rightarrow \cO(T^*Y_{\reg})\cong \cO(\overline{\bfO}_{n})
    \]
    is an isomorphism. Then we get the following commutative diagram:
    \begin{equation}
    \label{eq.diff-operators}
        \begin{tikzcd}[row sep=large, column sep=large]
        \gr \left(\cD(Y)^{\bC^*}\right)  \arrow[d,"{\gr(\pi_Y)_*}" left] \arrow[r]
            &  \cO(N)^{\bC^*}\cong \cO(M) \arrow[d] \\
        \gr \cD(\overline{\bO}_n) \arrow[r,"{\overline{\sigma}}"]
            & \cO(T^*\bO_n) 
    \end{tikzcd}
    \end{equation}
    Since $\bC^*$ is reductive, one gets an isomorphism $\gr\left(\cD(Y)^{\bC^*}\right)\cong \left(\gr \cD(Y)\right)^{\bC^*}$. Then the isomorphism $\overline{\sigma}_Y$ implies that the first row in \eqref{eq.diff-operators} is surjective. On the other hand, the right hand column in \eqref{eq.diff-operators} is an isomorphism by Theorem \ref{thm.main-thm}. As $\overline{\sigma}$ is injective, it follows that both $\gr(\pi_Y)_*$ and $\overline{\sigma}$ are surjective, as desired.
\end{proof}

\subsection{Divisor class group and $\bQ$-factorial terminalizations}

Firstly, we give a description of the Weil divisor class group and the Picard group of $\overline{T^*\bO_n}^{\aff}$, using the following general observation.

\begin{prop}
\label{p.Divgp}
    Let $Z$ be an irreducible affine normal variety such that $\overline{T^*Z_{\reg}}$ is an affine variety. Denote by $p\colon \overline{T^*Z_{\reg}}^{\aff}\rightarrow Z$ the natural projection. Then the induced morphisms
    \[
    \textup{WDiv}(Z)\longrightarrow \textup{WDiv}(\overline{T^*Z_{\reg}}^{\aff})\quad \text{and}\quad \Pic(Z) \longrightarrow \Pic(\overline{T^*Z_{\reg}}^{\aff})
    \]
    are isomorphisms.
\end{prop}

\begin{proof}
    Let $\bfZ\subset\overline{T^*Z_{\reg}}^{\aff}$ be the zero section. Then the induced morphism $\bfZ\rightarrow Z$ is an isomorphism. On the other hand, by \cite[Theorem 4.2]{Grosshans1997}, the natural morphism 
    \[
    T^*Z_{\reg}\longrightarrow \overline{T^*Z_{\reg}}^{\aff}
    \]
    is an open embedding with complement of codimension at least two, so the pull-back of Weil divisors are well-defined by restricting to $T^*Z_{\reg}\rightarrow Z_{\reg}$. Note that the following composition of homomorphisms 
    \[
    \text{WDiv}(Z)\xlongrightarrow{p^*} \text{WDiv}(\overline{T^*Z_{\reg}}^{\aff}) \cong \text{WDiv}(T^*Z_{\reg}) \cong\text{WDiv}(\bfZ_{\reg})\cong \text{WDiv}(\bfZ),
    \]
    is an isomorphism, where the isomorphism $\text{WDiv}(T^*Z_{\reg})\cong \text{WDiv}({\bfZ}_{\reg})$ is given by the restriction. Hence the pull-back $p^*$ is an isomorphism. 

    Similarly, for the Picard group, we consider the following isomorphism:
    \[
    \Pic(Z)\xlongrightarrow{p^*} \Pic(\overline{T^*Z_{\reg}}^{\aff}) \longrightarrow \Pic(\bfZ)
    \]
    Note that the second map is injective by the injectivity of the map between the corresponding Weil divisor class groups, therefore $p^*$ also induces an isomorphism of Picard groups.
\end{proof}

\begin{cor}
\label{cor.Q-factoriality}
    If $n\geq 3$, then $\textup{WDiv}(\overline{T^*\bO_n}^{\aff})\cong \bZ$ and $\Pic(\overline{T^*\bO_n}^{\aff})\cong\{e\}$. In particular, the affine closure $\overline{T^*\bO_n}^{\aff}$ is not $\bQ$-factorial for $n\geq 3$.
\end{cor}

\begin{proof}
    This follows directly from Proposition \ref{p.Divgp} and the fact that $\textup{WDiv}(\overline{\bfO}_n)\cong \bZ$ and $\Pic(\overline{\bO}_n)\cong \{e\}$ for $n\geq 3$.
\end{proof}

\begin{rem}
    Recall that the minimal nilpotent orbit closure $\overline{\bO}_2$ is isomorphic to $\bC^2/\langle \pm 1\rangle$ with $\textup{WDiv}(\overline{\bO}_2)\cong \bZ_2$ and $\Pic(\overline{\bO}_2)\cong \{e\}$. Thus $\overline{T^*\bO_2}^{\aff}$ is $\bQ$-factorial by Proposition \ref{p.Divgp}. 
\end{rem}

\begin{cor}
	\label{cor.Q-fact-ter}
	If $n\geq 3$, the symplectic $\overline{T^*\bO_n}^{\aff}$ admits only two non-isomorphic $\bQ$-factorial terminalizations.
\end{cor}

\begin{proof}
	By Theorem \ref{thm.main-thm} and \S\,\ref{ss.Q-fact-term}, there exist two non-isomorphic small $\bQ$-factorial terminalizations $M^+$ and $M^-$ of $\overline{T^*\bO_n}^{\aff}$. Moreover, by Corollary \ref{cor.Q-factoriality}, the relative Picard numbers of $M^+$ and $M^-$ are both equal to one and therefore $\overline{T^*\bO_n}^{\aff}$ has exactly two different $\bQ$-factorial terminalizations.
\end{proof}

\appendix	

\section{Omitted computational details}

In this appendix, we provide the detailed computations omitted from the main text.

\subsection{Proof of Proposition \ref{prop.Sing-M(C*)}}
\label{ss.prop.Sing-M(C*)}

Let $U$ be the open subset of $\overline{\bfO}_n\cap\{a_{ij}\not=0\}$, where $i\not=j$. Then $\xi\in U$ if and only if $\rank \xi\leq 2$, $\xi^2=0$ and $a_{ij}\not=0$. In particular, as $a_{ij}\not=0$ and $b_{ii}=c_{ii}=0$, then $\rank \xi\leq 2$ if and only if the rows $\bfr_{i+1}$ and $\bfr_{n+j+2}$ (resp. the columns $\bfl_{j+1}$ and $\bfl_{n+i+2}$) are linearly independent vectors such that the other row (resp. column) of $\xi$ can be written as a linear combination of them. Now $\xi^2=0$ can be reformulated as follows:
	\[
	\bfr_{i+1}\bfl_{j+1} = \bfr_{i+1}\bfl_{n+i+2} = \bfr_{n+j+2}\bfl_{j+1} = \bfr_{n+j+2} \bfl_{n+i+2} = 0.
	\]
	As $C=-C^t$ and $B=-B^t$, a straightforward computation yields
    \begin{equation}
    \label{eq.E1}
        \begin{dcases}
            \bfr_{i+1}\bfl_{j+1} = \bfr_{n+j+2}\bfl_{n+i+2}    = x_i u_j + \sum_{k=1}^n a_{ik} a_{kj} -  y_j v_i - \sum_{k=1}^n b_{ik} c_{jk} = 0 \\
		\bfr_{i+1} \bfl_{n+i+2} = 2\left(x_i v_i - \sum_{k=1}^n a_{ik}b_{ik}\right) = 0 \\
		\bfr_{n+j+2}\bfl_{j+1} =  2\left(-y_j u_j + \sum_{k=1}^n a_{kj} c_{jk}\right) = 0
        \end{dcases}
    \end{equation}
	Since $a_{ij}\not=0$ and $b_{ii}=c_{jj}=0$, it follows that $a_{ii}$, $b_{ij}$ and $c_{ji}$ are determined by the other terms in $\bfr_{i+1}$ and $\bfr_{n+j+2}$. As a conseqence, the variety $U$ is isomorphic to the open subset of 
    \[
    \bC^{4n-2}\coloneqq \{(\widehat{\bfr}_{i+1},\widehat{\bfr}_{n+j+2})\}
    \]
    defined by $a_{ij}\not=0$. where 
	\[
	\hat{\bfr}_{i+1}\coloneq (x_i,a_{i1},\dots,\hat{a}_{ii},\dots,a_{in},v_i,b_{i1},\dots,\hat{b}_{ii},\dots,\hat{b}_{ij},\dots,b_{in}) \in \bC^{2n-1}
	\]
	and
	\[
	\hat{\bfr}_{n+j+2} = (y_j,c_{j1},\dots,\hat{c}_{ji},\dots,\hat{c}_{jj},\dots,c_{jn},u_j,a_{1j},\dots,\hat{a}_{ij},\dots,a_{nj})\in \bC^{2n-1}.
	\]
    On the other hand, for $1\leq k\leq n$, we have
	\[
	a_{kk} = \frac{a_{ik}a_{kj}+b_{ik}c_{jk}}{a_{ij}}.
	\]
	One gets
	\begin{equation}
		\label{eq.N1}
		-a_{ij} \operatorname{Tr}(A) = \sum_{k=1}^n \left(a_{ik}a_{kj}+b_{ik}c_{jk}\right).
	\end{equation}
	In view of the first equality of \eqref{eq.E1}, as $a_{ij}\not=0$, the equation $\operatorname{Tr}(A)=0$ is then equivalent to 
	\begin{equation}
		\label{eq.Trace-Moment-map}
		-a_{ij}\operatorname{Tr}(A) = x_i u_j - v_i y_j - 2\sum_{k=1}^n b_{ik} c_{jk} = 0.
	\end{equation}

\subsection{Proof of Proposition \ref{prop.Sing-MZ2}}
\label{ss.prop.Sing-MZ2}

Let $U$ be the open subset $\overline{\bfO}_n\cap \{b_{ij}\not=0\}$; that is, $\xi\in U$ if and only if $\rank \xi\leq 2$, $\xi^2=0$ and $b_{ij}\not=0$ for some $1\leq i\not=j\leq n$. As $b_{ij}\not=0$ and $b_{ii}=b_{jj}=0$, then $\rank\xi\leq 2$ if and only if $\bfr_{i+1}$ and $\bfr_{j+1}$ (resp. $\bfl_{n+i+2}$ and $\bfl_{n+j+2}$) are linearly independent vectors and the other row (resp. column) of $\xi$ can be written as a linear combination of them. Now $\xi^2=0$ can be reformulated as follows:
\[
	\bfr_{i+1}\bfl_{n+i+2}=\bfr_{j+1}\bfl_{n+i+2}=\bfr_{j+1}\bfl_{n+i+2}=\bfr_{j+1}\bfl_{n+j+2}=0.
\]
As $B=-B^t$, a straightforward computation yields
\begin{equation}
\label{eq.E2}
    \begin{dcases}
    \bfr_{i+1}\bfl_{n+j+2} = \bfr_{j+1} \bfl_{n+i+2}  = x_i v_j - \sum_{k=1}^n a_{ik} b_{jk} + x_j v_i - \sum_{k=1}^n  a_{jk} b_{ik} = 0 \\
    \bfr_{i+1}\bfr_{n+i+2}  = 2\left(x_iv_i - \sum_{k=1}^n a_{ik} b_{ik}\right) = 0 \\
		\bfr_{j+1}\bfl_{n+j+2}   = 2\left( x_j v_j - \sum_{k=1}^n a_{jk} b_{jk} \right) = 0 
\end{dcases}
\end{equation}
On the other hand, for any $1\leq k\leq n$, we have
\[
a_{kk} = \frac{-a_{ik}b_{jk}+a_{jk}b_{ik}}{b_{ij}}
\]
which then implies
\begin{equation}
\label{eq.N2}
    b_{ij}\operatorname{Tr}(A)  = \sum_{k=1}^n \left(-a_{ik}b_{jk} + a_{jk} b_{ik}\right).
\end{equation}
If $b_{ij}\not=0$, then $b_{ij}\operatorname{Tr}(A)=0$ is equivalent to $\operatorname{Tr}(A)=0$. In particular, combining it with \eqref{eq.E2} shows that for any $\xi \in U_N\coloneqq U\cap N$, the terms $a_{ii}$, $a_{ij}$, $a_{ji}$ and $a_{jj}$ are determined by other terms in $\bfr_{i+1}$ and $\bfr_{j+1}$ as $b_{ij}=-b_{ji}\not=0$ and $b_{ii}=b_{jj}=0$. As a consequence, the variety $U_N$ is isomorphic to the open subset of
\[
	\bC^{4n-3}\coloneqq \left\{(\hat{\bfr}_{i+1},\hat{\bfr}_{j+1})\in \bC^{2n-1}\times \bC^{2n-1} \mid b_{ij}=-b_{ji}\right\},
\]
defined by $b_{ij}=-b_{ji}\not=0$, where 
	\[
	\hat{\bfr}_{i+1}\coloneqq (x_i,a_{i1},\dots,\hat{a}_{ii},\dots,\hat{a}_{ij},\dots a_{in}, v_i,b_{i1},\dots,\hat{b}_{ii},\dots,b_{in})\in \bC^{2n-1}
	\]
	and 
	\[
	\hat{\bfr}_{j+1}\coloneqq (x_j,a_{j1},\dots,\hat{a}_{ji},\dots,\hat{a}_{jj},\dots,a_{jn},v_j,b_{j1}\dots,\hat{b}_{jj},\dots,b_{jn})\in \bC^{2n-1}.
	\]

\subsection{Proof of Lemma \ref{lem.fibers-mu+}} 
\label{ss.fibers-mu+}

We will divide the proof into three different cases: $\bfx\not=0$, $\bfv\not=0$ and $B\not=0$.

Firstly, we consider the case where $\bfx\not=0$. Let $U(x_i)$ be the subset of $U$ consisting of $\xi$ with $x_i\not=0$. Then $\bfr_{i+1}$ and $\bfr_{n+2}$ (resp. $\bfl_1$ and $\bfl_{n+i+2}$) are linearly independent such that the other rows (resp. columns) can be written as a linear combination of them. A matrix $\xi$ of the form \eqref{eq.Xi-centralfiber} contained in $U(x_i)$ is determined by the following equations:
\[
\bfr_{i+1}  \bfl_1 = \bfr_{i+1} \bfl_{n+i+2} = \bfr_{n+2} \bfl_1 = \bfr_{n+2} \bfl_{n+i+2} = 0.
\]
Then an easy computation shows that for $\xi$ of the form \ref{eq.Xi-centralfiber}, the only non-trivial equation above is
\[
\bfr_{i+1} \bfl_{n+i+2} = 2x_i v_i=0.
\]
In particular, as $x_i\not=0$, we get $v_i=0$ and hence $U(x_i)$ can be identified to the subset of the vector space
\[
\bC^{2n-1} = \{(x_1,\dots,x_n,b_{i1},\dots,\hat{b}_{ii},\dots,b_{in})\}
\]
defined by $x_i\not=0$. Denote by $U_{\bfx}$ the closure of $U(x_i)$ in $N^{\text{ss}}$. As $U(x_j)\cap U(x_i)\not=\emptyset$ for any $j$, the variety $U_{\bfx}$ is also the closed closure of $U(x_j)$. On the other hand, since $v_i=0$ and the other row is a linear combination of $\bfr_i$ and $\bfr_{n+2}$, the variety $U_{\bfx}$ is a closed subset of dimension $2n-1$ contained in $N^{\textup{ss}}\cap \{\bfv=0\}$.

Next we consider the case where $\bfv\not=0$. Let $U(v_i)$ be the open subset of $U$ consisting of $\xi$ with $v_i\not=0$. As in the case $\bfx\not=0$, we can show that $U(v_i)$ is isomorphic to the subset of 
\[
\bC^{2n-1}=\{(v_1,\dots,v_n,b_{i1},\dots,\hat{b}_{ii},\dots,b_{in})\}
\]
defined by $v_i\not=0$. Denote by $U_{\bfv}$ the closure of $U(v_i)$ in $N^{\textup{ss}}$. Then $U_{\bfv}$ is independent of $i$ and it is contained in $N^{\textup{ss}}\cap \{\bfx=0\}$.

Now we consider the case where $B\not=0$. Let $U(b_{ij})$, $1\leq i<j\leq n$, be the open subset of $U$ consisting of $\xi$ with $v_i\not=0$. Then a matrix $\xi$ of the form \eqref{eq.Xi-centralfiber} contained in $U(b_{ij})$ is determined by the following equations:
\[
\bfr_{i+1}\bfl_{n+i+2} = \bfr_{i+1} \bfl_{n+j+2} = \bfr_{j+1} \bfl_{n+i+2} = \bfr_{j+1} \bfl_{n+j+2} = 0.
\]
Then one easily derives
\[
\begin{dcases}
	\bfr_{i+1} \bfl_{n+i+2} = 2 x_i v_i = 0 \\
	\bfr_{i+1} \bfl_{n+j+2} = \bfr_{j+1} \bfl_{n+i+2} = x_j v_i + v_j x_i = 0\\
	\bfr_{j+1} \bfl_{n+j+2} = 2x_j v_j =0.
\end{dcases}
\]
This implies $x_i=x_j=0$ or $v_i=v_j=0$. Consequently, the variety $U(b_{ij})$ consists of two irreducible components, $U(b_{ij})_{\bfx}$ and $U(b_{ij})_{\bfv}$, which are isomorphic to the locally closed subvarieties 
\[
\{(x_i,x_j,b_{i1},\dots,\hat{b}_{ii},\dots,b_{in},b_{j1},\dots,\hat{b}_{jj},\dots,b_{jn})\mid b_{ij}=-b_{ji}\not=0\} \subset \bC^{2n}
\]
and
\[
\{(v_i,v_j,b_{i1},\dots,\hat{b}_{ii},\dots,b_{in},b_{j1},\dots,\hat{b}_{jj},\dots,b_{jn})\mid b_{ij}=-b_{ji}\not=0\}\subset \bC^{2n},
\]
respectively. Denote by $U_{B,\bfx}$ and $U_{B,\bfv}$ the closures of $U(b_{ij})_{\bfx}$ and $U(b_{ij})_{\bfv}$, respectively, in $N^{\textup{ss}}$. Then $U_{B,\bfx}\subset N^{\textup{ss}}\cap \{\bfv=0\}$ and $U_{B,\bfv}\subset N^{\textup{ss}}\cap \{\bfx=0\}$. Moreover, for any $1\leq k<l\leq n$, as $U(b_{ij})_{\bfx}\cap U(b_{kl})_{\bfx}\not=\emptyset$ and $U(b_{ij})_{\bfv}\cap U(b_{kl})_{\bfv}\not=\emptyset$, the varieties $U_{B,\bfx}$ and $U_{B,\bfv}$ are both independent of $(i,j)$.

Finally, note that $U = U_{\mathbf{x}} \cup U_{\mathbf{v}} \cup U_{B,\mathbf{x}} \cup U_{B,\mathbf{v}}$, each of these four components has dimension $2n-1$ and $U_{\mathbf{x}} \neq U_{\mathbf{v}}$. Observe that $U_{B,\mathbf{x}} \cap U(x_i)$ and $U_{B,\mathbf{v}} \cap U(v_i)$ are non‑empty open subsets of $U_{B,\mathbf{x}}$ and $U_{B,\mathbf{v}}$, respectively. Consequently, $U_{B,\mathbf{x}} = U_{\mathbf{x}}$ and $U_{B,\mathbf{v}} = U_{\mathbf{v}}$, which establishes the desired result.

\subsection{Proof of Lemma \ref{lem.Fibers-center-fiber}}
\label{ss.lem.Fibers-center-fiber}
    Let $U$ be the open subset $\overline{\bfO}_n\cap \{u_i\not=0\}$; that is, $\xi \in U$ if and only if $\rank \xi\leq 2$, $\xi^2=0$ and $u_i\not=0$. One observes $\rank \xi \leq 2$ if and only if the rows $\bfr_1$ and $\bfr_{n+i+2}$ (resp. $\bfl_{i+1}$ and $\bfl_{n+2}$) are linearly independent vectors such that the other row (resp. column) of $\xi$ is a linear combination of them. On the other hand, the equation $\xi^2$ can be reformulated as follows:
    \[
    \bfr_1\bfl_{i+1}=\bfr_{1}\bfl_{n+2}=\bfr_{n+i+2}\bfl_{i+1}=\bfr_{n+i+2}\bfl_{n+2}=0.
    \]
    As $C=-C^t$, then a straightforward computation yields
    \begin{equation}
        \label{eq.E3}
        \begin{dcases}
        \bfr_1\bfl_{i+1} = \bfr_{n+i+2}\bfl_{n+2} = w u_i + \sum_{k=1}^n (u_k a_{ki} - v_k c_{ik}) = 0\\
        \bfr_1\bfl_{n+2} = - 2\sum_{k=1}^n u_k v_k = 0 \\
            \bfr_{n+i+2} \bfl_{i+1} =  - 2 \left(-y_i u_i + \sum_{k=1}^n a_{ki} c_{ik}\right) = 0
        \end{dcases}
    \end{equation}
 	On the other hand, for $1\leq j\leq n$, one gets
 	\[
 	a_{kk} = \frac{u_k a_{ki} + v_k c_{ik}}{u_i},
 	\]
 	which then implies
 	\begin{equation}
 		\label{eq.Fiber-eq2}
 		u_i\operatorname{Tr}(A) =  \sum_{k=1}^n\left( u_k a_{ki} + v_k c_{ik}\right).
 	\end{equation}
    Moreover, one easily derives that an element $\xi\in U$ is contained in $F_+\subset \{\bfx=\bfy=0\}$ if and only if 
    \[
    w=y_i=0.
    \]
    As a consequence, for a given $y\in (\bfu,\bfv)\in Y^-$ with $u_i\not=0$, combining \eqref{eq.E3} with $\operatorname{Tr}(A)=0$, the fiber $F_y$ of $p_+$ over $y$ is the closed subset of  
    \[
    (c_{i1},\dots,\hat{c}_{ii},\dots,c_{in},a_{1i},\dots,a_{ni})\in \bC^{2n-1}
    \]
    defined by \eqref{eq.Fy}.

	\bibliographystyle{alpha}
	\bibliography{SQ}
\end{document}